\documentclass[11pt,a4paper,oneside]{amsart}
\usepackage[a4paper,margin=2.5cm]{geometry}
\usepackage[dutch,english]{babel}
\usepackage{graphicx,caption,subcaption}                         
\usepackage[hyperfootnotes=false,colorlinks=true,linkcolor=blue,%
citecolor=purple,filecolor=magenta,urlcolor=cyan]{hyperref}
\usepackage{amsmath,amssymb,amscd,amsthm,amsfonts,anyfontsize,dsfont,%
enumerate,fix-cm,layout,lipsum,lpic,mathrsfs,mdwlist,pigpen,stmaryrd,%
tensor,thmtools,tikz,txfonts,xspace}
\usepackage[all]{xy}
\usepackage[oldstyle]{libertine}%
\usepackage[T1]{fontenc}
\usepackage[utf8]{inputenc}
\makeatletter
\renewcommand*\libertine@figurestyle{LF}
\makeatother
\usepackage[libertine,libaltvw,liby]{newtxmath}
\makeatletter
\renewcommand*\libertine@figurestyle{OsF}

\makeatother
\usepackage{mathtools}
\expandafter\def\csname ver@etex.sty\endcsname{3000/12/31}

\theoremstyle{plain}                          
\newtheorem{theorem}{Theorem}[section]
\newtheorem{proposition}[theorem]{Proposition}    

\newtheorem{corollary}[theorem]{Corollary}
\newtheorem{conjecture}[theorem]{Conjecture} 

\theoremstyle{definition}

\newtheorem{definition}[theorem]{Definition}

\theoremstyle{remark}
\newtheorem{remark}[theorem]{Remark}

\def\CP{\mathbb{C}\mathrm{P}^1}
 
\newcommand{\mb}[1]{\mathbb{#1}} 

\newcommand{\mc}[1]{\mathcal{#1}}

\newcommand{\cW}{\mathcal{W}} 
 
\newcommand{\N}{\mb{N}} 
\newcommand{\C}{\mb{C}} 
\newcommand{\Z}{\mb{Z}}

\def\oM{\overline{\mathcal{M}}}
\def\ch{\mathrm{ch}}
\def\Chiodo{\mathsf{C}}

\DeclareMathOperator*{\Res}{Res}
\DeclareMathOperator{\ev}{ev}
\DeclareMathOperator{\restr}{\Big\rfloor}

\newcommand{\Hol}{\mathord{\textup{Hol}}}
\newcommand{\cord}[1]{\bigg{\langle} \, #1 \, \bigg{\rangle}}
\newcommand{\p}{\mathbf{p}}

\begingroup\expandafter\expandafter\expandafter\endgroup
\expandafter\ifx\csname pdfsuppresswarningpagegroup\endcsname\relax
\else
\fi
{}

\numberwithin{equation}{section}

\begin{document}
\title{Loop equations and a proof of Zvonkine's $qr$-ELSV formula}

\author[P.~Dunin-Barkowski]{P.~Dunin-Barkowski}

\address[P.~Dunin-Barkowski]{Faculty of Mathematics, National Research University Higher School of Economics, Usacheva 6, 119048 Moscow, Russia; and ITEP, 117218 Moscow, Russia}
\email{ptdunin@hse.ru}

\author[R.~Kramer]{R.~Kramer}

\address[R.~Kramer]{Max-Planck-Institut f\"ur Mathematik, 
	Vivatsgasse 7,
	53111 Bonn, Germany}
\email{rkramer@mpim-bonn.mpg.de}

\author[A.~Popolitov]{A.~Popolitov}

\address[A.~Popolitov]{Department of Physics and Astronomy, Uppsala University, Uppsala, Sweden;
  Moscow Institute for Physics and Technology, Dolgoprudny, Russia;
  Institute for Information Transmission Problems, Moscow 127994, Russia;
  and ITEP, Moscow 117218, Russia}
\email{a.popolitov@physics.uu.se}

\author[S.~Shadrin]{S.~Shadrin}

\address[S. Shadrin]{Korteweg-de Vriesinstituut voor Wiskunde, 
	Universiteit van Amsterdam, Postbus 94248,
	1090GE Amsterdam, Netherlands}
\email{s.shadrin@uva.nl}

\begin{abstract}
	We prove the 2006 Zvonkine conjecture~\cite{Zvonkine2006} that expresses Hurwitz numbers with completed cycles in terms of intersection numbers with the Chiodo classes~\cite{Chio08} via the so-called $r$-ELSV formula, as well as its orbifold generalization, the $qr$-ELSV formula, proposed recently in~\cite{KLPS17}.  
\end{abstract}

\maketitle

\tableofcontents

\section{Introduction}


This paper is concerned with spin Hurwitz numbers, which have been conjectured by Zvon\-{}kine~\cite{Zvonkine2006} to be expressable as integrals over the moduli space of curves, in a generalized ELSV formula, called Zvonkine's $r$-ELSV formula. In~\cite{KLPS17}, the authors conjectured an orbifold generalization of this formula, called Zvonkine's $qr$-ELSV formula. In this paper, we prove the latter, and hence also the former, formula, via topological recursion and quadratic loop equations. We will introduce all of these concepts in this introduction.

\subsection{\texorpdfstring{$q$}{q}-orbifold \texorpdfstring{$r$}{r}-spin Hurwitz numbers}
In this section we introduce the $q$-orbifold $r$-spin Hurwitz numbers, following \cite{OkPa06a,Zvonkine2006,SSZ12,SSZ,KLPS17}. They are a very important and natural type of Hurwitz numbers; more precisely, they are a special case of completed Hurwitz numbers. Completed Hurwitz numbers were introduced by Okounkov and Pandharipande in \cite{OkPa06a} to establish a relation between Hurwitz numbers and relative Gromov-Witten invariants; in this section we recall their result specified for the $q$-orbifold $r$-spin case.

\subsubsection{Completed cycles}

A {\em partition}~$\lambda$ of an integer~$d$ is a non-increas\-ing
finite sequence $\lambda_1 \geq \dots \geq \lambda_l$ 
such that $\sum \lambda_i =d$. 

It is known that the irreducible representations $\rho_\lambda$
of the symmetric group $S_d$ are in a natural one-to-one
correspondence with the partitions~$\lambda$ of~$d$. 
On the other hand, to a partition~$\lambda$
of~$d$ one can assign a central element $C_{p,\lambda}$ 
of the group algebra $\C S_p$ for any positive integer~$p$. 
The coefficient of a given permutation $\sigma \in S_p$ in
$C_{p,\lambda}$ is defined as the number of ways to choose
and number $l$ cycles of $\sigma$ so that their lengths are 
$\lambda_1, \dots, \lambda_l$, and the remaining $p-d$ elements are fixed points of~$\sigma$. Thus the coefficient of $\sigma$ vanishes unless its cycle lengths are $\lambda_1, \dots, \lambda_l, 1, \dots, 1$. In particular, $C_{p, \lambda} = 0$ if $p < d$. Thus $C_{p, \lambda}$ is
{\em the sum of permutations with $l$ numbered cycles
	of lengths $\lambda_1, \dots, \lambda_l$ and any number of non-numbered fixed points}.

The collection of elements $C_{p, \lambda}$ for $p = 1, 2, \dots$
is called a {\em stable center element}
$C_\lambda$. 
For example, the stable element $C_2$ is
the sum of all transpositions in $\C S_p$, which is well-defined for each~$p$, and in particular equals zero for $p=1$.

Let~$\lambda$ be a partition of~$d$ and~$\mu$ a partition of~$p$. Since $C_{p,\lambda}$ lies in the center of $\C S_p$, 
it is represented by a scalar (multiplication by a constant) 
in the representation~$\rho_\mu$ of~$S_p$. Denote this scalar
by $f_\lambda(\mu)$. Thus to a stable center element $C_\lambda$ we have assigned a function $f_\lambda$
defined on the set of all partitions, $\mc{P}$. We are interested in the
vector space spanned by the functions~$f_\lambda$.

To study this space, one defines some new functions 
on the set of partitions as follows:
\begin{equation}
\p_{r+1}(\mu) = \frac1{r+1} \sum_{i \geq 1} 
\left[(\mu_i - i + \frac12)^{r+1} - (-i + \frac12)^{r+1} \right]
\qquad (r \geq 0).
\end{equation}
(The standard definition
	\cite[p.11]{OkPa06a}
	involves certain additive constants that we have dropped to
	simplify the expression, since these constants play no role in this paper.)

\begin{theorem}[Kerov, Olshansky \cite{KerovOlshanski1994}]
	The vector space spanned by the functions $f_\lambda$ coincides with the algebra generated by the functions $\p_1, \p_2, \dots$.
\end{theorem}

As a corollary, to each stable center element $C_\lambda$ we can assign a polynomial in $\p_1, \p_2, \dots$ and, conversely, each $\p_{r+1}$ corresponds to a linear combination of stable center elements $C_\lambda$.

\begin{definition} \label{Def:CompCyc}
	The linear combination of stable center elements corresponding
	to $\p_{r+1}$ is called the {\em completed $(r+1)$-cycle} and denoted
	by~$\overline{C}_{r+1}$.
\end{definition}

The first completed cycles are:
\begin{align}
\overline{C}_1&= C_1,\\ \nonumber
\overline{C}_2&= C_2,\\ \nonumber
\overline{C}_3&= C_3 + C_{1,1} + \frac1{12} C_1,\\ \nonumber
\overline{C}_4&= C_4 + 2 C_{2,1} + \frac54 C_2,\\ \nonumber
\overline{C}_5&= C_5 + 3 C_{3,1} + 4 C_{2,2} + \frac{11}3 C_3
+ 4 C_{1,1,1} + \frac32 C_{1,1} + \frac1{80} C_1.
\end{align}

We say that a stable center element $C_\lambda$ involved in the completed cycle
$\overline{C}_{r+1}$ has {\em genus defect} $[r+2 - \sum (\lambda_i + 1)]/2$.

\subsubsection{$r$-spin Hurwitz numbers}
Let $g\in\mathbb{Z}_{\geq 0}$ and $r\in\mathbb{Z}_{\geq 1}$. Let $\vec{\mu}=(\mu_1, \dots, \mu_n)$ be an integer partition of length $n=\ell(\mu)$ such that $m \coloneqq (\sum_{i=1}^n \mu_i + n + 2g-2)/r$ is an integer, and let $d \coloneqq |\mu|= \sum_{i=1}^n \mu_i$. 

Recall that the completed $(r+1)$-cycle can be considered as a central element of the group algebra $\C S_d$. An {\em $r$-factorization of type $(\mu_1, \dots, \mu_n)$} in the symmetric group $S_d$ is a  factorization
\begin{equation}
\sigma_1 \dots \sigma_m = \sigma
\end{equation}
such that
\begin{itemize}
\item[(i)] the cycle lengths of~$\sigma$ equal $\mu_1, \dots, \mu_n$ and
\item[(ii)] each permutation $\sigma_i$ enters the completed $(r+1)$-cycle
with a nonzero coefficient.
\end{itemize}
The product of these coefficients for $i$ going from $1$ to~$m$ is called the {\em weight} of the $r$-factorization.

Choose $m$ points $y_1, \dots, y_m \in \C$ and a system of $m$ loops $s_i \in \pi_1(\C \setminus \{ y_1, \dots, y_m \})$, $s_i$ going around $y_i$. Then to an $r$-factorization one can assign a family of stable maps  from nodal curves to $\mathbb{C}P^1$. This is done in the following way.
\begin{itemize}
\item[(i)] Consider the covering of $\mathbb{C}P^1$ ramified over $y_1, \dots, y_m$, and $\infty$ with monodromies given by $\sigma_1, \dots, \sigma_m$ and $\sigma^{-1}$ (relative to the chosen loops).
\item[(ii)] If $\sigma_i$ has $l_i$ distinguished cycles and genus  defect~$g_i$, glue a curve of genus $g_i$ with $l_i$ marked points to the $l_i$ preimages of the $i$-th ramification point that correspond to the distinguished cycles. The covering mapping is extended on this new component by  saying that it is entirely projected to the $i$-th ramification point.
\item[(iii)] Among the newly added components, contract those that are
unstable.
\end{itemize}
One can easily check that the arithmetic genus of the curve~$C$ constructed in this way is equal to~$g$. The complex structure on the newly added components of~$C$  can be chosen arbitrarily, which implies that in general we obtain not a unique stable map, but a family of stable maps. 

An $r$-factorization is called {\em transitive} if the curve~$C$ assigned to the factorization is connected, in other words if one can go from every element of $\{1, \dots, d \}$ to any other by applying the permutations $\sigma_i$ and jumping from one distinguished cycle of $\sigma_i$ to another one.

\begin{definition}
	The \emph{disconnected $r$-spin Hurwitz number} $h_{g;\vec{\mu}}^{\bullet,r}$ is the sum of weights of all $r$-factorizations of type $(\mu_1, \dots, \mu_n)$, divided by $|\mu |!m!$.
\end{definition}

\begin{definition}
	The \emph{connected $r$-spin Hurwitz number} $h_{g;\vec{\mu}}^{\circ,r}$ is the sum of weights of transitive $r$-fac\-{}to\-{}ri\-{}za\-{}tions of type $(\mu_1, \dots, \mu_n)$, divided by $|\mu|!m!$.
\end{definition}

Note that connected and disconnected $r$-spin Hurwitz numbers are related via the usual inclusion-exclusion formula.

\subsubsection{$q$-orbifold $r$-spin Hurwitz numbers}

The $q$-orbifold $r$-spin Hurwitz numbers arise as a generalization of the previous case, when one adds another ramification point with profile $[qq\dots q]$. In the language of the symmetric group this looks as follows.

Let $g\in\mathbb{Z}_{\geq 0}$, $r\in\mathbb{Z}_{\geq 1}$ and $q\in\mathbb{Z}_{\geq 1}$. Let $\vec{\mu}=(\mu_1, \dots, \mu_n)$ be an integer partition of length $n=\ell(\mu)$ such that $d \coloneqq |\mu|= \sum_{i=1}^n \mu_i$ is divisible by $q$ and $m \coloneqq (d/q + n + 2g-2)/r$ is an integer.

A {\em $q,r$-factorization of type $(\mu_1, \dots, \mu_n)$} in the symmetric group $S_d$ 
is a  factorization
\begin{equation}
\sigma_1 \dots \sigma_m \gamma= \sigma
\end{equation}
such that 
\begin{itemize}
\item[(i)] the cycle lengths of~$\gamma$ are all equal to $q$,
\item[(ii)] the cycle lengths of~$\sigma$ equal $\mu_1, \dots, \mu_n$ and
\item[(iii)] each permutation $\sigma_i$ enters the completed $(r+1)$-cycle with a nonzero coefficient.
\end{itemize}
The product of these coefficients for $i$ going from $1$ to~$m$
is called the {\em weight} of the $r$-factorization.

In a way completely analogous to the non-orbifold case we can define \emph{transitive} $q,r$-factorizations. Then we can proceed to defining disconnected and connected $q$-orbifold $r$-spin Hurwitz numbers:

\begin{definition}
	The \emph{disconnected $q$-orbifold $r$-spin Hurwitz number} $h_{g;\vec{\mu}}^{\bullet,q,r}$ is the sum of weights of all $q,r$-factorizations of type $(\mu_1, \dots, \mu_n)$, divided by $|\mu |!m!$.
\end{definition}

\begin{definition}
	The \emph{connected $q$-orbifold $r$-spin Hurwitz number} $h_{g;\vec{\mu}}^{\circ,q,r}$ is the sum of weights of transitive $q,r$-factorizations of type $(\mu_1, \dots, \mu_n)$, divided by $|\mu|!m!$.
\end{definition}

Again, connected and disconnected $q$-orbifold $r$-spin Hurwitz numbers are related via the usual inclusion-exclusion formula.

Naturally, for $q=1$ one recovers the $r$-spin Hurwitz numbers, for $r=1$ one recovers the $q$-orbifold Hurwitz numbers, while for $q=r=1$ one arrives at the classical simple Hurwitz numbers.

\subsubsection{Semi-infinite wedge formalism}\label{sec:semi-infwedge}

This subsection is devoted to writing $q$-orbifold $r$-spin Hurwitz numbers in terms of the \emph{semi-infinite wedge formalism} (also known as \emph{free-fermion formalism} to physicists).

First, we define the basic ingredients of this formalism.
For a more complete introduction see e.g. \cite{John15}. We will write \( \Z' \coloneqq \Z + \frac{1}{2}\) for the set of half-integers.
\begin{definition}
	The Lie algebra \( \mathcal{A}_\infty \) is the \( \C \)-vector space of matrices \( (A_{ij})_{i,j \in \Z'} \) with only finitely many non-zero diagonals, together with the commutator bracket.\par
	In this algebra, we will consider the following elements:
	\begin{enumerate}
		\item The standard basis of this algebra is the set \( \{ E_{i,j} \mid i,j \in \Z'\} \) such that \( (E_{i,j})_{k,l} = \delta_{i,k} \delta_{j,l} \);
		\item The diagonal algebra elements (operators) \( \mathcal{F}_n \coloneqq \sum_{k\in \Z'} k^n E_{k,k}\). In particular, \( C \coloneqq \mathcal{F}_0 \) is the \emph{charge operator} and \( E \coloneqq \mathcal{F}_1 \) is the \emph{energy operator}. An algebra element \( A\) has energy \( e\in \Z \) if \( [A,E] = eA \);
		\item For any non-zero integer \( n\), the energy \( n\) element \( \alpha_n \coloneqq \sum_{k \in \Z'} E_{k-n,k} \).
	\end{enumerate}
\end{definition}
The semi-infinite wedge space is a certain projective representation of this algebra, which we will construct now.
\begin{definition}
	Let \( V\) be the vector space spanned by \( \Z'\): \( V = \bigoplus_{i \in \Z'} \C \underline{i} \), where the \( \underline{i} \) are basis elements. We define the \emph{semi-infinite wedge space} \( \mathcal{V} \coloneqq \bigwedge^{\frac{\infty}{2}} V \) to be the span of all one-sided infinite wedge products
	\begin{equation}\label{eq:semi-infwedge}
	\underline{i_1} \wedge \underline{i_2} \wedge \dotsb, 
	\end{equation}
	with \( i_1 < i_2 < \dotsb \in \Z'\), such that there exists a constant $c$ with \( i_k + k -\frac12 = c\) for large \( k\). 
	The constant \( c\) is called the \emph{charge}.
\end{definition}
\begin{remark}
	Notice that \( \mathcal{A}_\infty \) has a natural representation on \( V\), but this cannot be extended to \( \mathcal{V} \) easily, as one would have to deal with infinite sums.
\end{remark}
\begin{definition}
	For a partition \( \lambda \), define
	\begin{equation}
	v_\lambda \coloneqq \underline{\lambda_1 -\frac12} \wedge \underline{\lambda_2 - \frac32}  \wedge \dotsb.
	\end{equation}
	In particular, define the vacuum \( |0\rangle \coloneqq v_\emptyset \) and let the covacuum \( \langle 0| \) be its dual in \( \mathcal{V}^* \).\par
	Define \( \mathcal{V}_0 \) to be the charge-zero subspace of \( \mathcal{V} \). Then \( \mathcal{V}_0 = \bigoplus_{\lambda \in \mc{P}} \C v_\lambda \).
\end{definition}
\begin{definition}
	For an endomorphism \( \mc{O} \) of \( \mathcal{V}_0 \), define its \emph{vacuum expectation value} or \emph{disconnected correlator} to be
	\begin{equation}
	\langle \mc{O} \rangle^\bullet \coloneqq \langle 0 | \mc{O} | 0 \rangle.
	\end{equation}
\end{definition}
\begin{definition}
	Define a projective representation of \( \mathcal{A}_\infty \) on \( \mathcal{V}_0\) as follows: for \( i \neq j\) or \( i = j > 0\), \( E_{i,j} \) checks whether \( v_\lambda \) contains \( \underline{j} \) as a factor and replaces it by \( \underline{i} \) if it does. If \( i = j < 0\), \( E_{i,i} v_\lambda = -v_\lambda \) if \( v_\lambda \) does not contain \( \underline{j} \). In all other cases it gives zero.\par
	Equivalently, this gives a representation of the central extension \( \tilde{\mathcal{A}}_\infty = \mathcal{A}_\infty \oplus \C 1 \), with commutation between basis elements
	\begin{equation}\label{eq:CommEE}
	\left[E_{a,b}, E_{c,d}\right] = \delta_{b,c} E_{a,d} - \delta_{a,d} E_{c,b} +  \delta_{b,c}\delta_{a,d}(\delta_{b>0} - \delta_{d>0})1 .
	\end{equation}
\end{definition}
With these definitions, it is easy to see that \( C \) is identically zero on \( \mathcal{V}_0 \) and \( Ev_\lambda = |\lambda |v_\lambda \). Therefore, any positive-energy operator annihilates the vacuum. Similarly, so do all \( \mc{F}_r \).\par

The \( q\)-orbifold \( r\)-spin Hurwitz numbers can be represented as vacuum expectations of certain operators.
We will write \( \mu = a[\mu ]_a + \langle \mu \rangle_a \) for the integral division of an integer \( \mu \) by a natural number \( a\). If \( a = qr\), we may omit the subscript.

The $q$-orbifold $r$-spin Hurwitz numbers can be represented in terms of the semi-infinite wedge formalism as described in the following proposition.
\begin{proposition}
	The disconnected \( q\)-orbifold $r$-spin Hurwitz numbers can be expressed in terms of semi-infinite wedge formalism as
	\begin{equation}\label{eq:defHurwcorr}
	h_{g;\vec{\mu}}^{\bullet, q,r} =  \cord{  \Big(\frac{\alpha_q}{q}\Big)^{\frac{|\mu |}{q}} \frac{1}{\big(\frac{|\mu |}{q}\big)!} \, \frac{ \mathcal{F}_{r+1}^m}{m!(r+1)^m} \, \prod_{i=1}^{l(\vec{\mu})}  \frac{\alpha_{-\mu_i}}{\mu_i} },
	\end{equation}
	where the number of $(r+1)$-completed cycles is 
	\begin{equation}\label{eq:b-NumberMarkedPoints}
	m = \frac{2g - 2 + l(\mu) + \frac{|\mu|}{q}}{r}.
	\end{equation}
\end{proposition}
This statement follows from the basic character formula for general Hurwitz numbers, see \cite{OkPa06a}.

\begin{definition}
	The \emph{generating series of \( q\)-orbifold \( r\)-spin Hurwitz numbers} is defined as
	\begin{equation}
	H^{\bullet,q, r}(\vec{\mu}, u) \coloneqq \sum_{g=0}^{\infty} h_{g;\vec{\mu}}^{\bullet, q,r} u^{rm} = \cord{  e^{\frac{\alpha_q}{q}}\, e^{u^r\frac{\mathcal{F}_{r+1}}{r+1}}\, \prod_{i=1}^{l(\vec{\mu})} \frac{\alpha_{-\mu_i}}{\mu_i} }.
	\end{equation}
	The \emph{free energies} are defined as
	\begin{equation}\label{FreeEnergies}
	H_{g,n}^{q,r}(X_1,\dotsc, X_n) \coloneqq \sum_{\mu_1, \dotsc, \mu_n =1}^\infty h_{g;\vec{\mu}}^{\circ, q,r} e^{\sum_{i=1}^n \mu_i X_i}
	\end{equation}
\end{definition}

With the help of semi-infinite wedge formalism, in \cite{KLPS17} the following \emph{quasi-polynomiality theorem} was proved in a purely combinatorial way:
\begin{theorem}[\cite{KLPS17}]\label{thm:r-spinpoly}
	For \( 2g-2+ \ell (\vec{\mu} ) > 0\), the connected \( q\)-orbifold \( r\)-spin Hurwitz numbers can be expressed in the following way:
	\begin{equation}
	h_{g,\vec{\mu}}^{\circ,q, r} = \prod_{i=1}^{l(\vec{\mu})} \frac{\mu_i^{[\mu_i]}}{[\mu_i]!}
	P_{\langle \vec{\mu} \rangle}(\mu_1, \dots, \mu_{l(\vec{\mu})}),
	\end{equation}
	where $P$ are symmetric polynomials in the variables $\mu_1, \dots, \mu_{l(\vec{\mu})}$ whose coefficients depend on the parameters $\langle \mu_1 \rangle, \dots, \langle \mu_{l(\vec{\mu})} \rangle$, and which has an upper bound on its degree in all variables that is independent of~$\vec{\mu}$.
\end{theorem} 

\subsubsection{Relative Gromov-Witten invariants and the Okounkov-Pandharipande formula}

The $q$-orbifold $r$-spin Hurwitz numbers were originally introduced in~\cite{OkPa06a} because of their relation to relative Gromov-Witten invariants of $\mathbb{C}P^1$; this relation is a special case of the Okounkov-Pandharipande theorem from \cite{OkPa06a}, which we would like to recall.

Let $\oM_{g,m; \mu_1, \dots, \mu_n; q}\left(\mathbb{C}P^1\right)$ be the space of stable genus~$g$ maps to $\mathbb{C}P^1$ relative to $\{\infty,0\} \in \mathbb{C}P^1$ with profiles $(\mu_1, \dots, \mu_n)$ and $(q,q,\dots,q)$ respectively and with $m$ marked points in the source curve, where $m = (|\mu|/q + n + 2g-2)/r$. Let $[\oM_{g,m; \mu_1, \dots, \mu_n; q}\left(\mathbb{C}P^1\right)]^{\mathrm{vir}}$ be its virtual fundamental class.
See e.g.~\cite{Vakil2008} for the precise definition and main properties. Let $\omega \in H^2(\mathbb{C}P^1)$ be the Poincar\'e dual class of a point.

A special case of Okounkov--Pandharipande theorem from \cite{OkPa06a} states that
\begin{theorem}[Okounkov--Pandharipande, \cite{OkPa06a}] \label{thm:OkoPan}
	Connected $q$-orbifold, $r$-spin Hurwitz numbers are equal to certain relative Gromov-Witten invariants of $\mathbb{C}P^1$. Specifically, we have:
	\begin{equation}
	h_{g,\vec{\mu}}^{\circ,q, r} = \frac{(r!)^m}{m!}
	\int\limits_{[\oM_{g,m; \mu_1, \dots, \mu_n; q}\left(\mathbb{C}P^1\right)]^{\mathrm{vir}}} 
	\!\!\!\!
	\ev_1^*(\omega) \psi_1^r \cdots \ev_m^*(\omega)\psi_m^r 
	\end{equation}
	Here $\ev_i$ denotes the evaluation map $\oM_{g,m; \mu_1, \dots, \mu_n; q}\left(\mathbb{C}P^1\right) \rightarrow \mathbb{C}P^1$ at the $i$-th marked point, $i=1,\dots,m$, and $\psi_i~\in~H^2\left(\oM_{g,m; \mu_1, \dots, \mu_n; q}\left(\mathbb{C}P^1\right)\right)$ is the $\psi$-class corresponding to the $i$-th marked point.
\end{theorem}


{\subsection{Chiodo classes and Zvonkine's conjecture} \label{sec:chiodo-classes}
	The central objects in Zvonkine's conjecture are the so-called Chiodo classes, which are cohomology classes on the moduli spaces of stable curves $\overline{\mc{M}}_{g,n}$.
	In this section we briefly recall their definition, as well as properties relevant for our proof.
	More details can be found in \cite{Chio08,ChiodoRuan,JPPZ,SSZ,KLPS17,CladerJanda}.
	
	{\subsubsection{Geometric definition} \label{sec:chiodo-classes-geomdef}
		{ 
			Let $r\geq 1$ be an integer and $ g\geq 0$, $n \geq 1$, $1\leq a_1,\dots,a_n\leq r$, and $s \geq 0$ be integers satisfying 
			\begin{equation} \label{eq:root-existence-condition}
				(2g-2+n)s-\sum_{i=1}^n a_i \in r\Z
			\end{equation}
		}
		
		{ 
			Let $[C, p_1, \dots, p_n] \in \mathcal{M}_{g,n}$ be a nonsingular curve with distinct marked points.
			Furthermore, let $\omega_{\text{log}} = \omega_C (\sum p_i)$ be its log-canonical bundle.
			The condition \eqref{eq:root-existence-condition} ensures that $r$th tensor roots $L$ of the line bundle
			\begin{equation}
				\omega_{\text{log}}^{\otimes s} \left ( -\sum a_i p_i \right )
			\end{equation}
			on $C$ exist. There is a natural compactification of this moduli space of $r$th roots, 
			denoted $\overline{\mathcal{M}}^{r,s}_{g;a_1,\dots,a_n}$,
			which is an analog of the Deligne-Mumford compactification of $\mathcal{M}_{g,n}$ and was constructed in
			\cite{Chiodo-stable-twisted,Jarvis,AbramovichJarvis,CaporaseEtAl}.
		}
		
			Let $\pi : \mathcal{C}^{r,s}_{g;a_1,\dots a_n} \rightarrow \overline{\mathcal{M}}^{r,s}_{g;a_1,\dots a_n}$ be the universal curve
			and let $\mathcal{L} \rightarrow \mathcal{C}^{r,s}_{g;a_1,\dots a_n}$
			be the universal $r$th root.
		%
		%
			The Chiodo class is the full Chern class of the derived push-forward
				$c(- R^\bullet \pi_* \mathcal{L})$.

		{ 
			In practice, we only need an expression for the pushforward of the Chiodo class
			to the compactified moduli space of curves $\overline{\mathcal{M}}_{g, n}$.
			There is an explicit formula for this pushforward in terms of tautological classes,
			which we recall below.
		}
	}
	
	{\subsubsection{Formula in terms of tautological classes} \label{sec:chiodo-classes-tau-formula}
		The Chern characters of the derived push-forward $R^\bullet \pi_* {\mathcal L}$ are given by Chiodo's formula~\cite{Chio08}. In order to give this formula, we first need to give some definitions. For any nodal curve in $\overline{\mathcal{M}}^{r,s}_{g;a_1,\dots a_n}$, the nodes must have automorphism group $ \Z /r\Z$, inducing a primitive character on the cotangent line at each side of the branch (we pick one side). The line bundle $L$ at this side is naturally a $\Z /r\Z$-representation, because it is an $r$-th root. This representation is then an $a$-th power of the representation of the cotangent line at the point for some $a$. This $a$ is locally constant on the boundary divisor, and hence we can split this divisor into components. We let $j_a$ be the boundary map for the $a$-th component. We also write $\psi',\psi''$ for the $\psi$-classes at the two branches of the node (in general, we use standard notation for $\psi$ and $\kappa$ tautological classes, see e.g.~\cite{Vakil2003,Zvonkine2012}). Then Chiodo's formula is
		\begin{align} \label{eq:chiodo-tau-formula}
			\ch_m(R^\bullet \pi_* {\mathcal L}) = 
			\frac{B_{m+1}(\frac sr)}{(m+1)!} \kappa_m
			- \sum_{i=1}^n 
			\frac{B_{m+1}(\frac{a_i}r)}{(m+1)!} \psi_i^m 
			+ 
			\frac{r}2 \sum_{a=0}^{r-1} 
			\frac{B_{m+1}(\frac{a}r)}{(m+1)!} (j_a)_* 
			\frac{(\psi')^m + (-1)^{m-1} (\psi'')^m}{\psi'+\psi''}\,.
		\end{align}
		The Bernoulli polynomials $B_{l}(x)$ used in this formula are generated by the function
		\begin{align} \label{eq:bernoulli-generating-function}
			\sum_{l=0}^\infty B_l(x) \frac{t^l}{l!} = \frac{t e^{x t}}{e^t - 1}\,.
		\end{align}
		Let $\epsilon$ be the forgetful map
		\begin{equation}
			\epsilon : \overline{\mathcal{M}}^{r,s}_{g;a_1,\dots,a_n} \rightarrow \overline{\mathcal{M}}_{g, n}
		\end{equation}
		We are interested in the pushforwards of the Chiodo classes
		\begin{align} \label{eq:chiodo-def-cohft}
			\Chiodo_{g,n}(r,s;a_1,\dots,a_n) := & \ \epsilon_* c(-R^\bullet \pi_* {\mathcal L}) = \epsilon_* \left[ c(R^1\pi_*{\mathcal L})/c(R^0\pi_*{\mathcal L})\right]
			\\ \notag
			= & \ \epsilon_{*}\exp\left(\sum_{m=1}^\infty (-1)^m (m-1)!\ch_m(R^\bullet\pi_* {\mathcal L})\right) \in H^{\rm even}(\oM_{g,n}).
		\end{align}
The pushforwards of the Chiodo classes form a  \emph{cohomological field theory} in the sense of~\cite{KontsevichManin} (with non-flat unit if $s>r$), and can therefore be written explicitly in terms of the Givental graphs, see~\cite{LPSZ-Chiodo}.
	
	{\subsubsection{Zvonkine's $qr$-ELSV formula} \label{sec:chiodo-classes-lpsz-hypothesis}
		In \cite{KLPS17} the authors proposed the following conjecture, which is a direct orbifold generalization of Zvonkine's conjecture.
		\begin{conjecture}\cite[Conjecture 6.1]{KLPS17} \label{conj:main} $q$-orbifold $r$-spin Hurwitz numbers
			are given by the formula
			\begin{align}
				h_{g,\mu_1,\dots,\mu_n}^{\circ,q, r} = 
				r^{2g-2+n}(qr)^{\frac{(2g-2+n)q+\sum_{j=1}^n \mu_j}{qr}}  \prod_{j=1}^{n} \frac{\big(\frac{\mu_j}{qr}\big)^{[\mu_j]}}{[\mu_j]!}
				\int_{\oM_{g,n}} \frac{\Chiodo_{g,n} \left (qr,q; qr- \left < \mu_1 \right >,
					\dots,
					qr - \left < \mu_n \right > \right )
				}{\prod_{j=1}^n (1-\frac{\mu_i}{qr}\psi_i)} \,,
			\end{align}
		where $\mu=qr[\mu]+ \left < \mu \right >$ is the integral division of $\mu$ by $qr$.
		\end{conjecture}
		
		This conjecture expresses the $q$-orbifold $r$-spin Hurwitz numbers as an explicit ELSV-like integral over the moduli space of curves, where the role of the Hodge class $1-\lambda_1+\cdots\pm \lambda_g $ is played by the pushforward of the Chiodo class,
		$\Chiodo_{g,n}(r,s;a_1,\dots,a_n)$. We call this formula for the $q$-orbifold $r$-spin Hurwitz numbers Zvonkine's $qr$-ELSV formula. 
		
		This conjecture is already known for $q=r=1$ (in this case it is the standard ELSV formula proved in~\cite{ELSV01}, see also~\cite{GraberVakil,DKOSS13}),  $r=1$, $q\geq 1$ (then it is the Johnson-Pandharipande-Tseng formula proved in~\cite{JohnsonPandharipandeTseng}, see also~\cite{DLPS}), and $r=2$, $q\geq 1$ (proved in~\cite{SpecialCases}). It is also known to hold for any $q,r\geq 1$  in genus $g=0$~\cite{SpecialCases}.
		
		The main result of this paper is a proof of conjecture~\ref{conj:main} in full generality:
		
		\begin{theorem}\label{thm:ZvonkineTrue} Zvonkine's $qr$-ELSV formula holds. 
		\end{theorem}
		
		The proof of this theorem uses the formalism of CEO topological recursion explained below. Let us note one more fact before proceeding to that.	Namely, our main result, theorem \ref{thm:ZvonkineTrue}, together with Okounkov--Pandharipande's theorem (theorem \ref{thm:OkoPan}) immediately imply the following purely intersection theory statement
		\begin{corollary}	\label{cor:intersect}		
			\begin{align}
			&\frac{(r!)^m}{m!} \int\limits_{[\oM_{g,m; \mu_1, \dots, \mu_n; q}\left(\mathbb{C}P^1\right)]^{\mathrm{vir}}} 
			\!\!\!\!
			\ev_1^*(\omega) \psi_1^r \cdots \ev_m^*(\omega)\psi_m^r \\ \nonumber &= 
			\int_{\oM_{g,n}} \frac{\Chiodo_{g,n} \left (qr,q; qr- \left < \mu_1 \right >,
				\dots,
				qr - \left < \mu_n \right > \right )
			}{\prod_{j=1}^n (1-\frac{\mu_i}{qr}\psi_i)}
			\cdot
			r^{2g-2+n}(qr)^{\frac{(2g-2+n)q+\sum_{j=1}^n \mu_j}{qr}}  \prod_{j=1}^{n} \frac{\big(\frac{\mu_j}{qr}\big)^{[\mu_j]}}{[\mu_j]!}\,.
			\end{align}
		\end{corollary}

\subsection{Topological recursion}

\subsubsection{General setup}
The topological recursion of Chekhov, Eynard, and Orantin~\cite{ChEy06,EyOr07,EynardSurvey} 
associates to a Riemann surface $\Sigma$ (the so-called spectral curve) equipped with two functions $X,y\colon \Sigma\to \mathbb{C}$ and a symmetric bidifferential $B$ on $\Sigma^2$ satisfying some extra conditions a family of meromorphic symmetric $n$-differentials (CEO-differentials) $\omega_{g,n}$ defined on $\Sigma^n$, $g\geq 0$, $n\geq 1$. We assume that $dX$ is meromorphic and all critical points $p_1,\dots,p_r$ of $X$ are simple, $y$ is holomorphic near $p_i$ and $dy\not=0$ at $p_i$, $i=1,\dots,r$, and $B$ has no singularities except for a double pole on the diagonal with biresidue $1$. We set by definition $\omega_{0,1} = ydX$, $\omega_{0,2} = B$, and for $2g-2+n>0$ we define:
\begin{align}\label{eq:CEO-toprec}
\omega_{g,n}(z_{\{1,\dots,n\}}) = 
\frac 12 \sum_{i=1}^r \Res_{z\to p_i} \frac{\int_z^{\sigma_i(z)}\omega_{0,2}(\cdot,z_1)}{\omega_{0,1}(\sigma_i(z))-\omega_{0,1}(z)} \Big[
\omega_{g-1,n+1}(z,\sigma_i(z),z_{\{2,\dots,n\}}) + 
\\ \notag 
\sum_{\substack{g_1+g_2 = g \\ I_1\sqcup I_2 = \{2,\dots,n\} \\ (g_i,|I_i|)\not= (0,0)}}
\omega_{g_1,1+|I_1|}(z,z_{I_1})\omega_{g_2,1+|I_2|}(\sigma_i(z),z_{I_2})
\Big].
\end{align}
Here $\sigma_i$ is the deck transformation for $X$ near the point $p_i$, $i=1,\dots,r$, and all $\omega_{-1,n}$, $n\geq 1$, are set to be equal to $0$. Furthermore, for a set $I$, we write $ z_I = \{ z_i \}_{i \in I}$.

Eynard proved in~\cite{Eynard2014} that for $2g-2+n>0$ the meromorphic differentials $\omega_{g,n}$ can be represented as linear combinations of the intersection numbers of some explicitly computed tautological classes on $\oM_{g,n}$ multiplied by some auxiliary differentials. Under some extra conditions, see~\cite{DOSS14} and also~\cite{DBNOPS-Dubrovin,DBNOPS-Primary}, it is proved in~\cite{DOSS14} that the meromorphic differentials $\omega_{g,n}$ can be represented in terms of the correlators of a semi-simple cohomological field theory of rank $r$, where the cohomological field theory is given explicitly in terms of Givental graphs~\cite{DBSS-GivGraphs}, and some other auxiliary differentials. More precisely, for $2g-2+n>0$ the differentials $\omega_{g,n}$ are represented as
\begin{equation}\label{eq:CohFT-differential}
\omega_{g,n}=\sum_{\substack{i_1,\dots,i_n \\ a_1,\dots,a_n}} \int_{\overline{\mathcal{M}}_{g,n}} \alpha_{g,n}(e_{i_1},\dots,e_{i_n}) \prod_{j=1}^n
\psi_j^{a_j} d\left(\left(\frac{d}{dX}\right)^{a_j} \xi^{i_j}(z_j)\right),
\end{equation}
where
\begin{equation}\label{eq:xi-function}
\xi^i(z)\coloneqq \int^z \left. \frac{\omega_{0,2} (w_i,\cdot)}{dw_i}\right|_{w_i=0}\,
\end{equation}
for a local coordinate $w_i$ near $ p_i$, and $\alpha_{g,n}\colon V^{\otimes n} \to H^*(\oM_{g,n},\mathbb{C})$ form a cohomological field theory, where $V$ is an $r$-dimensional vector space with basis $\langle e_1,\dots,e_r\rangle$.

\subsubsection{Particular spectral curves}
We consider the spectral curve data
\begin{equation}\label{eq:SpectralCurveData}
\Sigma=\CP,\ X(z)=-z^{qr}+\log z,\ y(z)=z^q,\ B(z_1,z_2)= dz_1dz_2/(z_1-z_2)^2.
\end{equation}
It is more convenient to work with this curve using the function $x=e^X = ze^{-z^{qr}}$. 
For this curve all the ingredients of the formula in equation~\eqref{eq:CohFT-differential} can be computed explicitly, and it is proved in~\cite{LPSZ-Chiodo} that the expansions of $\omega_{g,n}$ in the variables $x_1,\dots,x_n$ near $x_1=\cdots = x_n= 0$ are given by
\begin{align}
\omega_{g,n} \sim\  d_1\otimes \cdots \otimes d_n \sum_{\mu_1,\dots,\mu_n=1}^\infty &
\int_{\oM_{g,n}} \frac{\Chiodo_{g,n} \left (rq,q; qr- \left < \mu_1 \right >,
	\dots,
	qr -  \left < \mu_n \right > \right )
}{\prod_{j=1}^n (1-\frac{\mu_i}{qr}\psi_i)}
\\ \notag &
\cdot
r^{2g-2+n}(qr)^{\frac{(2g-2+n)q+\sum_{j=1}^n \mu_j}{qr}}  \prod_{j=1}^{n} \frac{\big(\frac{\mu_j}{qr}\big)^{[\mu_j]}}{[\mu_j] !}x_j^{\mu_j}.
\end{align}
Thus we have the following proposition.

\begin{proposition}[\cite{LPSZ-Chiodo,SSZ}] \label{prop:equiva} Zvonkine's $qr$-ELSV formula holds if and only if the expansion of the CEO-differentials $\omega_{g,n}$ for the curve~\eqref{eq:SpectralCurveData} in  the variables $x_1,\dots,x_n$ near $x_1=\cdots = x_n= 0$ is given by
\begin{equation}\label{eq:expansionCEOdiff}
\omega_{g,n} - \delta_{g,0}\delta_{n,2} \frac{dx_1dx_2}{(x_1-x_2)^2} \sim\  d_1\otimes \cdots \otimes d_n \sum_{\mu_1,\dots,\mu_n=1}^\infty h_{g;\mu}^{\circ, q,r} \prod_{i=1}^n x_i^{\mu_i}.
\end{equation}
\end{proposition}

Thus, an equivalent way to reformulate  theorem~\ref{thm:ZvonkineTrue} is
\begin{theorem}\label{thm:toporec} The expansion of the CEO-differentials  $\omega_{g,n}$ for the curve~\eqref{eq:SpectralCurveData} in  the variables $x_1,\dots,x_n$ near $x_1=\cdots = x_n= 0$ is given by equation~\eqref{eq:expansionCEOdiff}.
\end{theorem}

\begin{remark} The spectral curve for the $q$-orbifold $r$-spin Hurwitz numbers in full generality was predicted in~\cite{MSS} via the analysis of the so-called quantum curve. 
\end{remark}

\begin{remark}
Historically, this theorem was first formulated for $q=r=1$ as the Bouchard-Mari\~no conjecture~\cite{BouchardMarino}, and this case was first proved in~\cite{EMS-Laplace} \emph{using the ELSV formula for Hurwitz numbers}, see also~\cite{Eyna11}. In a similar way, this theorem was proved for any $q$, $r=1$ in~\cite{BouchardMulase,Do-Orbifold} \emph{using the Johnson-Pandharipande-Tseng formula}. These proofs are not exactly what we want, since we want to use the inverse of their arguments, namely, we want to use this theorem in order to prove Zvonkine's $qr$-ELSV formula.	
\end{remark} 

\begin{remark} Proofs independent of Zvonkine's $qr$-ELSV formula are known in special cases. First of all, there are non-rigorous physics arguments in~\cite{BorotMulase} for $q=r=1$ and in~\cite{SSZ} for $q=1$, any $r$. Then there are rigorous proofs in~\cite{DKOSS13} for $q=r=1$, in~\cite{DLPS} for any $q$, $r=1$ (see also~\cite{KLS} for an alternative argument for a part of that proof, and a discussion in~\cite{LewanskiSurvey}), and in~\cite{SpecialCases} for any $q$, $r=2$.  This theorem is also already known for any $q, r\geq 1$ in genus $g=0$, see~\cite{KLPS17} for the unstable cases $n=1,2$ and~\cite{SpecialCases} for $n\geq 3$.
\end{remark}

\subsubsection{Loop equations} We use a reformulation of the CEO topological recursion proved in~\cite{BEO13,BoSh15}. We say that a system of meromorphic differentials $\omega_{g,n}$ with possible poles at $p_1,\dots,p_{qr}$ satisfies the \emph{projection property} if $P_{1}\cdots P_{n} \omega_{g,n} = \omega_{g,n}$ for $2g-2+n>0$, where for any meromorphic differential $\lambda$ we define
\begin{equation}
(P \lambda)(z) = \sum_{j=1}^{qr} \Res_{w\to p_i}  \lambda(w) \int_{p_i}^w \omega_{0,2} (\cdot, z),
\end{equation}
and by writing $P_i$ we mean that we apply this operation to the $i$-th variable.

Denote
\begin{equation}
W_{g,n}(z_{\{1,\dots,n\}}) := \omega_{g,n}(z_{\{1,\dots,n\}}) / \prod_{j=1}^n dX(z_j).
\end{equation}

We say that a system of meromorphic differentials $\omega_{g,n}$ with possible poles at $p_1,\dots,p_{qr}$ satisfies the \emph{linear loop equations} if for any $g\geq 0$, $n\geq 1$ the expression
\begin{equation}\label{eq:LLE1}
W_{g,n}(z,z_{\{2,\dots,n\}}) + W_{g,n}(\sigma_i(z),z_{\{2,\dots,n\}})
\end{equation}
is holomorphic in $z$ for $z\to p_i$, $i=1,\dots,qr$.

We say that a system of meromorphic differentials $\omega_{g,n}$ with possible poles at $p_1,\dots,p_{qr}$ satisfies the \emph{quadratic loop equations} if for any $g\geq 0$, $n\geq 0$ the expression
\begin{equation}\label{eq:QLE1}
W_{g-1,n+2}(z,\sigma_i(z),z_{\{1,\dots,n\}}) + \sum_{\substack{g_1+g_2 = g \\ I_1\sqcup I_2 = \{1,\dots,n\} }}
{W}_{g_1,1+|I_1|}(z,z_{I_1}){W}_{g_2,1+|I_2|}(\sigma_i(z),z_{I_2})
\end{equation}
is holomorphic in $z$ for $z\to p_i$, $i=1,\dots,qr$. 

\begin{proposition}[\cite{BEO13,BoSh15}] \label{prop:equivCEOloop} A system of meromorphic differentials $\omega_{g,n}$ with $\omega_{0,1} = ydX$, $\omega_{0,2} = B$, satisfies the CEO topological recursion for the data $(\Sigma,X,y,B)$ if and only if it satisfies the projection property, the linear loop equation, and the quadratic loop equation, where point $p_i$ are the cricial points of map $X$.
\end{proposition}

\subsubsection{Quasi-polynomiality}

There is one property that is crucial for our proof scheme of the $q$-Zvonkine conjecture: the so-called {\em quasi-polynomiality}. For $q$-orbifold $r$-spin Hurwitz numbers this qua\-{}si-po\-{}ly\-{}no\-{}mi\-{}a\-{}li\-{}ty is given in theorem~\ref{thm:r-spinpoly}, proved in~\cite{KLPS17}. Using~\cite[lemma~4.6]{SSZ}, theorem~\ref{thm:r-spinpoly} is equivalent to the following statement:
\begin{proposition}\label{prop:FreeEnergyXiFunctions}
For $2g-2+n>0$ the free energies of equation~\eqref{FreeEnergies} are expansions of finite linear combinations of functions of the shape
\begin{equation}
 \prod_{j=1}^n \Big(\frac{d}{dX}\Big)^{a_j} \xi^{i_j}(z_j)
 \end{equation}
 with the $\xi^i$ defined by equation~\eqref{eq:xi-function} for the spectral curve data given by equation~\eqref{eq:SpectralCurveData}. 
\end{proposition}
\begin{remark}
Under the change $ X \to x$, we get
\begin{equation}
H_{g,n}^{q,r}(x_1,\dotsc, x_n) = \sum_{\mu_1, \dotsc, \mu_n =1}^\infty h_{g;\vec{\mu}}^{\circ, q,r} \prod_{i=1}^n x_i^{\mu_i}\,.
\end{equation}
We will often omit the superscripts $q$ and $r$.
\end{remark}

For more background on the importance of quasi-polynomiality, we refer the interested reader to~\cite{LewanskiSurvey}.

Relating this proposition to equations~\eqref{eq:CohFT-differential} and~\eqref{eq:expansionCEOdiff}, we see that the free energies have the `right shape' to satisfy topological recursion. In particular, proposition~\ref{prop:FreeEnergyXiFunctions} implies the free energies can be interpreted as functions defined globally on the curve~\eqref{eq:SpectralCurveData} rather than formal power series. We will use this viewpoint from now on.

The operator of the derivative $\frac d {dX} = x\frac d{dx}$ is denoted by $D_x$.

Since the functions $d(D_x)^a \xi^i$, $i=1,\dots, r$, $a=0,1,2,\dots$, satisfy the projection property, that is, $P d (D_x)^a \xi^i = d (D_x)^a \xi^i $, and the linear loop equation, that is, $d(D_x)^a \xi^i (z) + d(D_x)^a \xi^i (\sigma_j(z))$ is holomorphic for $z\to p_j$, $j=1,\dots,qr$,  we have:
\begin{proposition} The system of meromorphic differentials $d_1\otimes \cdots \otimes d_n H_{g,n}$  satisfies the projection property and the linear loop equations.
\end{proposition}
\begin{remark}\label{rem:Hlinloop}
	Note that proposition \ref{prop:FreeEnergyXiFunctions} also implies that for $2g-2+n>0$ the $n$-point functions  $H_{g,n}$ themselves,  once one puts them onto the spectral curve, satisfy a property similar to the linear loop equations. Namely, the sum $H_{g,n}(z,z_{\{2,\dots,n\}}) + H_{g,n}(\sigma_i(z),z_{\{2,\dots,n\}})$ is holomorphic at the $i$-th ramification point. This also follows from the fact that $(D_x)^a \xi^i (z) + (D_x)^a \xi^i (\sigma_j(z))$ is holomorphic for $z\to p_j$, $j=1,\dots,qr$.
\end{remark}

Thus theorem~\ref{thm:toporec} is a corollary of proposition~\ref{prop:equivCEOloop} and the following statement, whose proof is the technical core of this paper:

\begin{theorem}\label{thm:QuadraticLoop} The system of meromorphic differentials $d_1\otimes \cdots \otimes d_n H_{g,n}$ on the curve~\eqref{eq:SpectralCurveData} satisfies the quadratic loop equations.
\end{theorem} 

The rest of this paper is a proof of this theorem (reformulated as theorem~\ref{thm:QuadraticLoopForSpinHurwitz} below), which is derived from the analysis of implications of the quadratic loop equations and their comparison with the so-called \emph{cut-and-join equation} for the $r$-spin Hurwitz numbers. The cut-and-join equation for the $r$-spin Hurwitz numbers was proved in~\cite{SSZ12}, see also~\cite{Rossi,Alexandrov}, and converted in the form that we use in this paper in~\cite{SpecialCases}.

\subsection{Further remarks} Though we tried to make this paper as self-contained as possible, the full proof of Zvonkine's conjecture from scratch includes several big steps performed in~\cite{SSZ12},~\cite{SSZ},~\cite{LPSZ-Chiodo}, \cite{KLPS17}, and~\cite{SpecialCases}, and they are absolutely necessary for our proof. In particular, some familiarity with \cite{SpecialCases} may be very helpful to follow the technical steps of the proof below. 

Our proof is definitely not of the kind that closes the whole area of research. For instance, neither the geometric interpretation of spin Hurwitz numbers in terms of relative Gromov-Witten invariants of $\CP$ (recalled in theorem~\ref{thm:OkoPan} above), nor the geometric definition of the Chiodo classes and/or geometry of the moduli space of $r$-th roots (see section~\ref{sec:chiodo-classes-geomdef} above) played any role in the argument. We hope that a geometric proof of Zvonkine's conjecture (in the form of corollary \ref{cor:intersect}) will be found (for instance,  some ideas are discussed in a recent preprint~\cite{Leigh}). 

Finally, we would like to mention that a quite general framework for topological recursion for Hurwitz numbers was recently proposed by Alexandrov, Chapuy, Eynard, and Harnad in~\cite{AlChapuy}. The spectral curve data~\eqref{eq:SpectralCurveData} is a special case of their proposal, while the $r$-spin Hurwitz numbers seem not to fit into their formalism (cf.~the discussion of quantum curves in~\cite{ALS}). It does not lead to any immediate contradiction, since the proof in~\cite{AlChapuyProof} does not cover the cases we are interested here, but it would be extremely interesting to unify the point of view of~\cite{AlChapuy} with the results of the present paper.

\subsection{Acknowledgements}
We would like to thank G.~Borot, B.~Bychkov, M.~Kazarian, D.~Lewa\'{n}ski, L.~Spitz, and D.~Zvonkine for stimulating discussions and the anonymous referee for useful remarks. We also thank Maxim Kazarian for pointing out several gaps in a previous version of the paper.
P.~D.-B. and A.~P. also would like to acknowledge the warm hospitality of Korteweg-de Vries Institute
for Mathematics.
P.~D.-B. was supported by the Russian Science Foundation (project 16-11-10316).
R.~K. and S.~S. were supported by the Netherlands Organization for Scientific Research.
A.~P. was supported in part by
Vetenskapsr\r{a}det under grant \#2014-5517,
by the STINT grant,
by the grant  ``Geometry and Physics"  from the Knut and Alice Wallenberg foundation,
and by the RFBR grants 18-31-20046 mol\_a\_ved and 19-01-00680 A.

\section{The cut-and-join equation and quadratic loop equations}
Let $ \llbracket n \rrbracket \coloneqq \{ 1, \dotsc, n\}$. The spin cut-and-join equation, \cite[equation (17)]{SpecialCases}, is

\begin{equation}\label{cutresummed}
\frac{B_{g,n}}{r!}\tilde{H}_{g,n}(x_{\llbracket n \rrbracket})  = \!\!\!\! \sum_{\substack{m \geq 1, d \geq 0 \\ m + 2d = r+1}} \!\! \frac{1}{m!} \; \sum_{l=1}^{m}\; \sum_{\substack{\{ k\} \sqcup \bigsqcup_{j=1}^\ell K_j = \llbracket n \rrbracket\\ \bigsqcup_{j=1}^\ell M_j =  \llbracket m \rrbracket\\ \forall j\; M_j \neq \emptyset}}  \! \frac{1}{l!} \!\sum_{\substack{g_1,\ldots,g_{\ell} \geq 0 \\ g = \sum_j g_j + m - \ell + d}} \!\!\!\!\! Q_{d,\emptyset,m}^{(k)} \bigg[\prod_{j = 1}^{\ell} \tilde{H}_{g_j,|M_j|+|K_j|}(\xi_{M_j},x_{K_j})\bigg]\,.
\end{equation}
Here, $B_{g,n} \coloneqq \frac{1}{r} \big(2g-2+n+ \frac{1}{q} \sum_{i=1}^n D_{x_i} \big) $ and
\begin{equation}
\label{Qrdef} \sum_{d \geq 0} Q_{d;\emptyset ,m}^{(k)}\,z^{2d} = \frac{z}{\zeta(z)} \frac{\zeta(zD_{x_k})}{zD_{x_k}} \circ \prod_{j = 1}^m \frac{\zeta(zD_{\xi_j})}{z}\bigg|_{\xi_j = x_k},
\end{equation}
$\zeta(z)\coloneqq e^{z/2}-e^{-z/2}$.
Furthermore,
\begin{align}
\tilde{H}_{0,1} &\coloneqq H_{0,1} & \\ \nonumber
\tilde{H}_{0,2}(\xi_1,\xi_2) &\coloneqq H_{0,2}(\xi_1,\xi_2) & \\ \nonumber
\tilde{H}_{0,2}(\xi_1,x_2) &\coloneqq H_{0,2}(\xi_1,x_2) + H_{0,2}^{{\rm sing}}(\xi_1,x_2)\,&H_{0,2}^{{\rm sing}}(\xi_1,x_2) \coloneqq \log \Big(\frac{\xi_1 - x_2}{\xi_1 x_2}\Big)\\ \nonumber
\tilde{H}_{g,n} &\coloneqq H_{g,n} -\delta_{2g-2+n,r}r! \frac{(2^{1 - 2g} - 1)B_{2g}}{2g!}, & 2g-2+n>0.
\end{align}
\begin{remark}
Note that we have abused the notation above, defining $\tilde{H}_{0,2}(\xi_1,\xi_2)$ differently from $\tilde{H}_{0,2}(\xi_1,x_2)$, such that these two objects are different depending on whether they have two $\xi$-variables or one $\xi$- and one $x$-variable as arguments. This is a necessary evil, as otherwise the formulas would become very bulky.
\end{remark}

This formula may seem rather daunting, so let us give some examples for small $r$. First, we calculate
\begin{align}
Q^{(k)}_{0;\emptyset,m} &= \prod_{j = 1}^m D_{\xi_j} \bigg|_{\xi_j = x_k}\,;\\ \nonumber
Q^{(k)}_{1;\emptyset,m} &= \frac{1}{24} \bigg( \Big( D_{x_k}^2 -1 \Big) \circ \Big( \prod_{j=1}^m D_{\xi_j} \bigg|_{\xi_j = x_k}\Big) + \sum_{l = 1}^m \prod_{j=1}^m D_{\xi_j}^{1 + 2\delta_{j,l}} \bigg|_{\xi_j =k} \bigg)\,.
\end{align}
In the 'non-spin' case, $r=1$, the first sum only includes the summand for $m =2$, $d=0$, so the formula reduces to
\begin{equation}
B_{g,n} \tilde{H}_{g,n}(x_{\llbracket n \rrbracket})  = \frac{1}{2} \sum_{k=1}^n    D_{\xi_1}D_{\xi_2}\bigg[ \tilde{H}_{g-1,n+1}(\xi_1, \xi_2, x_{\llbracket n\rrbracket \setminus \{ k\}} )\, + \!\!\!\!\!\!\!\!\!\!\! \sum_{\substack{g_1 + g_2 = g\\ K_1 \sqcup K_2 = \llbracket n \rrbracket \setminus \{ k\}}} \!\!\!\!\!\!\!\!\!\! \tilde{H}_{g_1, |K_1|+1}(\xi_1, x_{K_1}) \tilde{H}_{g_2,|K_2|+1}(\xi_2, x_{K_2}) 
\bigg]\bigg|_{\xi_1 = \xi_2 = x_k}\,.
\end{equation}
A full derivation of this formula from the standard cut-and-join equation is available in~\cite[section 3.3]{DKOSS13}. This equation should be interpreted as describing the removal of a transposition (completed $2$-cycle) from a $2$-factorization. Geometrically, this means removing a ramification point with simple ramification (partition $(2, 1^{d-2})$). After removing this, the two sheets which were glued together before either still belong to one connected curve, of genus one less (the linear term on the right-hand side) or now belong two two different curves (the quadratic term). Notice that in this equation the factor $\frac{1}{l!}$ cancels the overcounting coming from the decompositions of $ \llbracket n \rrbracket $ and $ \llbracket m\rrbracket $, which always give $l!$ identical terms.\par
In the case $r=2$, we get either $m=3 $ and $d=0$ or $ m=1$ and $d=1$. Hence (cf.~\cite[equation (23)]{SpecialCases}),
\begin{align}
B_{g,n}\tilde{H}_{g,n}(x_{\llbracket n \rrbracket})  =&  \frac{1}{3} \sum_{k=1}^n \Big( D_{\xi_1} D_{\xi_2} D_{\xi_3} \tilde{H}_{g-2,n+2} (\xi_1,\xi_2,\xi_3, x_{\llbracket n\rrbracket \setminus \{ k\}} ) \Big) \bigg|_{\xi_1 = \xi_2 = \xi_3 = x_k}\\ \nonumber
& +  \sum_{\substack{g_1 + g_2 = g-1\\ \nonumber \{ k\} \sqcup K_1 \sqcup K_2 = \llbracket n \rrbracket}} \Big( D_{x_k} \tilde{H}_{g_1, |K_1\ +1}(x_k, x_{K_1}) \Big) \Big( D_{\xi_1} D_{\xi_2} \tilde{H}_{g_2,|K_2|+2}(\xi_1,\xi_2,x_{K_2})\Big)\bigg|_{\xi_1 = \xi_2 = x_k}\\ \nonumber
 &+ \frac{1}{3} \sum_{\substack{g_1 + g_2 +g_3 = g\\ \nonumber \{ k\} \sqcup  K_1 \sqcup K_2 \sqcup K_3 = \llbracket n \rrbracket}} \prod_{j = 1}^3 D_{x_k} \tilde{H}_{g_j,|K_j|+1}(x_k,x_{K_j}) \\ \nonumber
&+ \frac{1}{12} (2D_{x_k}^3- D_{x_k}) \tilde{H}_{g-1,n}(x_{\llbracket n \rrbracket })\,.
\end{align}
As in the case before, the terms are related to removing a cycle from the $3$-factorization, and considering the number of connected components of the resulting curve. A detailed exposition of the resulting combinatorics is available in~\cite[section~5.2]{SSZ12}. Because the completed $3$-cycle is not equal to the non-completed $3$-cycle, we get terms for each of the possible cycles to be removed, with extra coefficients. This is also what occurs for general $r$.

Our goal is to express equation \eqref{cutresummed} in terms of $z$ variables (coordinates on the curve), and take the sum of this equation and its local conjugate in $x_1$ near any of the ramification points of $x$. For notational simplicity, let us actually take the $(g,n+1)$ case of this equation, with added variable $x_0$, in which we symmetrize, and let us write $ \bar{w} = \sigma_i (w)$. Let us also apply the operator $D_{x_1}\cdots D_{x_n}$ to both sides of the equation. Then the left hand side becomes holomorphic by the linear loop equations and remark~\ref{rem:Hlinloop}, and the right hand side becomes (up to terms, again holomorphic due to the linear loop equations and remark \ref{rem:Hlinloop}) equal to 
\begin{equation}\label{RHSofCJsymm}
\begin{split}
\sum_{\substack{m \geq 1, d \geq 0 \\ m + 2d = r+1}} \!\! \frac{1}{m!} \; \sum_{l=1}^{m}\; \sum_{\substack{\bigsqcup_{j=1}^\ell K_j = \llbracket n \rrbracket\\ \bigsqcup_{j=1}^\ell M_j = \llbracket m \rrbracket\\ \forall j\; M_j \neq \emptyset}}  \! \frac{1}{l!} \!\sum_{\substack{g_1,\ldots,g_{\ell} \geq 0 \\ g = \sum_j g_j + m - \ell + d}} \!\!\!\!\! \bar{Q}_{d,m}(z_0)\bigg[\prod_{j = 1}^{\ell} \tilde{W}_{g_j, |M_j|+|K_j|}(w_{M_j},z_{K_j})\bigg] +\\
\sum_{\substack{m \geq 1, d \geq 0 \\ m + 2d = r+1}} \!\! \frac{1}{m!} \;\sum_{l=1}^{m}\; \sum_{\substack{\bigsqcup_{j=1}^\ell K_j = \llbracket n \rrbracket \\ \bigsqcup_{j=1}^\ell M_j = \llbracket m \rrbracket\\ \forall j \; M_j \neq \emptyset}}  \! \frac{1}{l!} \!\sum_{\substack{g_1,\ldots,g_{\ell} \geq 0 \\ g = \sum_j g_j + m - \ell + d}} \!\!\!\!\! \bar{Q}_{d,m}(z_0)\bigg[\prod_{j = 1}^{\ell} \tilde{W}_{g_j, |M_j|+|K_j|}(\bar{w}_{M_j},z_{K_j})\bigg] \,.
\end{split}
\end{equation}
Here, we use the notation
\begin{align}
\tilde{W}_{g,m+n} (w_{\llbracket m\rrbracket},z_{\llbracket n\rrbracket}) &\coloneqq \prod_{j=1}^m D_{\xi(w_j)}  \prod_{i=1}^n D_{x(z_i)} \, \tilde{H}_{g,m+n}(\xi(w_{\llbracket m \rrbracket}),x(z_{\llbracket n \rrbracket})) \\
\sum_{d \geq 0} \bar{Q}_{d,m}(z_0) t^{2d} &\coloneqq \frac{t }{\zeta(t)} \frac{\zeta(tD_{x(z_0)})}{tD_{x(z_0)}} \circ \prod_{j = 1}^m \bigg( \restr_{w_j = z_0} \circ \frac{\zeta(tD_{x(w_j)})}{tD_{x(w_j)}}\bigg)\\ \label{eq:OpRestrDef}
\restr_{w=z} F(w) &\coloneqq \Res_{w=z} F(w)\frac{dX(w)}{X(w)-X(z)}\,.
\end{align}
Although it was not stated explicitly in~\cite{SpecialCases}, the operator $\restr_{w=z}$ of setting two variables equal should be defined via the previous residue formula, as it is the analytic continuation of the corresponding operator in coordinates $x$ in the cut-and-join equation. Note that it might be more natural to define  the operator $\restr_{w=z}$ as $\Res_{w=z} F(w)\frac{dx(w)}{x(w)-x(z)}$. The difference between this operator and  \eqref{eq:OpRestrDef} is not important when we apply it to a function that has no pole on the diagonal (which is the case in all statements in the rest of the paper), but in particular computations \eqref{eq:OpRestrDef} appears to be more convenient, cf. the proof of proposition~\ref{AllDeltasHolomorphic}. 

In order to simplify this a bit more, define the $m$-disconnected, $ n$-connected correlators $\tilde{\cW}_{g,m,n}(w_{\llbracket m\rrbracket} \mid z_{\llbracket n \rrbracket}) $ (cf.~\cite{BouchardEynard}) by keeping only those terms in the inclusion-exclusion formula where each factor contains at least one \( w \):
\begin{equation}\label{ConnDisconnCorrelator}
\tilde{\cW}_{g,m,n}(w_{\llbracket m \rrbracket} \mid z_{\llbracket n \rrbracket}) \coloneqq \sum_{l=1}^{m}\;\sum_{\substack{\bigsqcup_{j=1}^\ell K_j = \llbracket n \rrbracket\\ \bigsqcup_{j=1}^\ell M_j = \llbracket m \rrbracket \\ \forall j\; M_j \neq \emptyset}}  \! \frac{1}{l!} \!\sum_{\substack{g_1,\ldots,g_{\ell} \geq 0 \\ g = \sum_j g_j + m - \ell}} \! \prod_{j = 1}^{\ell} \tilde{W}_{g_j,|K_j| + |M_j|}(z_{K_j},w_{M_j})\,.
\end{equation}
(The factor $\frac{1}{l!} $ is just a symmetry factor.) This is defined in such a way that $ \tilde{\cW}_{g,1,n}(z \mid z_{\llbracket n \rrbracket}) = \tilde{W}_{g,n+1}(z,z_{\llbracket n \rrbracket}) $ and $ \tilde{\cW}_{g,n,0}(z_{\llbracket n \rrbracket}\mid \emptyset) $ is the disconnected correlator. The genus $g$ here stands for the genus of all terms after all $ m$ $w_j$-points are glued to an $(m+1)$-pointed sphere. Then we get for the right-hand side of the symmetrized cut-and-join equation
\begin{equation}\label{CutJoinConcise}
\sum_{\substack{m \geq 1, d \geq 0 \\ m + 2d = r+1}} \!\! \frac{1}{m!} \, \bar{Q}_{d,m}(z_0) \Big( \tilde{\cW}_{g-d,m,n}(w_{\llbracket m \rrbracket} \mid z_{\llbracket n \rrbracket}) +\tilde{\cW}_{g-d,m,n}(\bar{w}_{\llbracket m\rrbracket} \mid z_{\llbracket n \rrbracket}) \Big)
\end{equation}

\section{Proof of the quadratic loop equations via the symmetrized cut-and-join equation}

For the rest of the paper, we fix a ramification point $p$ of $x$ and let $z \mapsto \bar{z}$ be the local deck transformation.

\begin{definition}
Define the \emph{symmetrizing operator} $\mc{S}_z$ and the \emph{anti-symmetrizing operator} $\Delta_z$ by
\begin{align}
\mc{S}_zf(z) &\coloneqq f(z) + f(\bar{z})\,;\\ \nonumber
\Delta_zf(z) &\coloneqq f(z) - f(\bar{z})\,.
\end{align}
\end{definition}

We use the identity
\begin{equation}\label{Sondiagonal}
\mc{S}_zf(\underbrace{z,\dotsc, z}_{r \text{ times}}) = 2^{1-r} \!\!\! \sum_{\substack{I \sqcup J = \llbracket r \rrbracket \\ |J|\,\,\text{even}}} \Big( \prod_{i \in I} \mc{S}_{z_i} \Big) \Big( \prod_{j \in J} \Delta_{z_j} \Big)f(z_1,\dotsc, z_r) \Big|_{z_i = z}\,,
\end{equation}
which was also used in \cite{SpecialCases}.\par

\subsection{Symmetrization and anti-symmetrization of the regularized $W_{0,2}$}

The main difficulty of the proof comes from the diagonal poles of $ \tilde{W}_{0,2}$, so it is useful to give explicit formulae for it. In the global coordinate $z$ we have~\cite[theorem~5.2]{KLPS17}: 
\begin{align}
\tilde{W}_{0,2} (z,w) &=\frac{1}{X'(z)\,X'(w)\, (z-w)^2} \,;\\ \nonumber
\tilde{W}_{0,2} (w_1,w_2) &= \frac{1}{X'(w_1)\, X'(w_2)\, (w_1-w_2)^2} - \frac{x(w_1)\,x(w_2)}{(x(w_1)-x(w_2))^2}\,.
\end{align}
Recall that in the cut-and-join equation, we need to use different formulas for $\tilde{W}_{0,2} $ if it has one $w$ and one $z$ as arguments (then it is the usual $W_{0,2}$) and if it has two $w$'s as arguments (in this case we use the regularized $W_{0,2}$). The latter is the one that can cause problems with diagonal poles. Hence, we should consider the action of $\mc{S}$ and $\Delta $ on $\tilde{W}_{0,2} (w_1,w_2)$, to simplify many of the terms. As our spectral curve only has simple ramifications, we can work in the local coordinate $z$ defined by $ X-X(p) = z^2/2$, so the involution is $ \bar{z} = -z$.
\begin{align}
\tilde{W}_{0,2} (w_1,w_2) &= \frac{1}{w_1\,w_2\,(w_1+w_2)^2} + \textup{holom}\,;\\ \nonumber
\mc{S}_{w_1}\mc{S}_{w_2}\tilde{W}_{0,2} (w_1,w_2) &= \frac{2 }{w_1\,w_2\,(w_1+w_2)^2} - \frac{2 }{w_1\,w_2\,(w_1-w_2)^2} +\textup{holom} = - \frac{2 }{(X(w_1)-X(w_2))^2}+ \textup{holom}\,;\\ \nonumber
\mc{S}_{w_1}\Delta_{w_2} \tilde{W}_{0,2} (w_1,w_2) &= \textup{holom}\,;\\ \nonumber
\Delta_{w_1} \Delta_{w_2} \tilde{W}_{0,2} (w_1,w_2) &= \frac{2 }{w_1\,w_2\,(w_1+w_2)^2} + \frac{2 }{w_1\,w_2\,(w_1-w_2)^2} +\textup{holom}\,.
\end{align}
From this, it follows that any combination containing $\mc{S}_{w_1}\Delta_{w_2} \tilde{W}_{0,2} (w_1,w_2) $ is holomorphic. Note also that a simple residue argument implies that once $\Delta_{w_1}\Delta_{w_2} \tilde{W}_{0,2} (w_1,w_2) $ is used in an expression holomorphic in $w_1$ and $w_2$ near $w_1=w_2=0$ and symmetric under the involution in both variables, the application of the operator  to the whole expression  $\restr_{w_1=w_2}$ retains holomorphicity despite  its poles on the diagonal $w_1-w_2=0$ and on the antidiagonal $w_1+w_2=0$.

In fact, in order to simplify the calculation a bit, we will redefine
\begin{align}
\widetilde{\mc{S}_{w_1}\mc{S}_{w_2}} \tilde{W}_{0,2} (w_1,w_2) \coloneqq {\mc{S}_{w_1}\mc{S}_{w_2}} \tilde{W}_{0,2} (w_1,w_2) + \frac{2 }{(X(w_1)-X(w_2))^2}\,;\\ 
\widetilde{\Delta_{w_1} \Delta_{w_2}} \tilde{W}_{0,2} (w_1,w_2) \coloneqq {\Delta_{w_1} \Delta_{w_2}} \tilde{W}_{0,2} (w_1,w_2) \, 
-\frac{2 }{(X(w_1)-X(w_2))^2}\,,\label{DeltaDeltaBreg}
\end{align}
i.e., during analysis of the RHS of~\eqref{CutJoinConcise}, after we have written the expression in terms of $\mc{S}$ and $\Delta$
{\em symbolically}, we do the said redefinition. It is clear that it does not change the expression --- it just regroups some terms.

Then the $\widetilde{\mc{S}_{w_1}\mc{S}_{w_2}} \tilde{W}_{0,2} (w_1,w_2)$ is holomorphic,
and we need only concern ourselves with $ \tilde{W}_{0,2} (w_1,w_2)$ with two $\Delta$'s acting on them. From now on, we will use these modified definitions of $\mc{S}\mc{S}$ and $\Delta\Delta$, and omit the tildes from notation.

\subsection{Formal corollaries of the quadratic loop equations}
From \eqref{eq:QLE1}, the $(g,n)$ quadratic loop equation states that 
\begin{equation}
\restr_{w = w_0} \tilde{\cW}_{g,2,n}(w_0, \bar{w} \mid z_{\llbracket n \rrbracket} ) \text{ is holomorphic in $w_0$ near ramification points.}
\end{equation}
Let us call $2g-2+n=:-\chi$ the \emph{negative Euler characteristic} of a given quadratic loop equation.

Note that due to the symmetry of $\tilde{W}_{g,2,n}$ in its first two arguments, the expression above can be rewritten as follows:
\begin{equation}
\restr_{w = w_0} \tilde{\cW}_{g,2,n}(w_0, \bar{w} \mid z_{\llbracket n \rrbracket} )=- \dfrac{1}{4}\,\restr_{w = w_0}\left( \Delta_{w_0} \Delta_w - \mc{S}_{w_0}\mc{S}_w \right) \tilde{\cW}_{g,2,n}(w_0,w\mid z_{\llbracket n\rrbracket}).
\end{equation}
Note that 
\begin{equation}
\restr_{w = w_0}\mc{S}_{w_0}\mc{S}_w \tilde{\cW}_{g,2,n}(w_0,w\mid z_{\llbracket n\rrbracket}) 
\end{equation}
is holomorphic due to the linear loop equation, see \eqref{eq:LLE1}, and thus the quadratic loop equation can be reformulated as the statement that
\begin{equation}\label{eq:QLEDeltaFormulation}
\restr_{w = w_0} \Delta_{w_0} \Delta_w \tilde{\cW}_{g,2,n}(w_0,w\mid z_{\llbracket n\rrbracket}) \text{ is holomorphic in $w_0$ near ramification points.}
\end{equation}

Now let us extend the quadratic loop equation onto $\tilde\cW_{g,m,n}$ for $m>2$. Namely, we have the following:
\begin{proposition}\label{prop:gen2loop}	
	Suppose that a set of functions $ (\tilde W_{g,n} )_{g,n}$ satisfies the quadratic loop equations up to negative Euler characteristic $-\chi$. Then, we get for any $s,\, g,\, n \geq 0$ such that $2g-2+n\leq -\chi$, that
	\begin{equation}
	\restr_{w = w_0} \Delta_{w_0} \Delta_w \tilde{\cW}_{g,2+s,n}(w_0,w,w_{\llbracket s \rrbracket}\mid z_{\llbracket n\rrbracket}) \text{ is holomorphic in $w_0$ near ramification points.}
	\end{equation}
\end{proposition}
\begin{proof}
With the help of the definition \eqref{ConnDisconnCorrelator}, it is easy to see that 
\begin{align} \label{eq:gen2loop}
&\tilde{\cW}_{g,2+s,n}(w_0,w,w_{\llbracket s \rrbracket}\mid z_{\llbracket n\rrbracket}) = \tilde{\cW}_{g-s,2,n+s}(w_0,w\mid w_{\llbracket s \rrbracket}, z_{\llbracket n\rrbracket})\\&\phantom{==}+\sum_{\substack{K_1\bigsqcup K_2 = \llbracket n \rrbracket\\ M_1\bigsqcup M_2 = \llbracket s \rrbracket \\ M_2\neq \emptyset}}\;\;\sum_{\substack{g_1,\, g_2\geq 0\\ \nonumber g_1+g_2=g-|M_1|}}\tilde{\cW}_{g_1,2,|M_1|+|K_1|}(w_0,w\mid w_{M_1}, z_{K_1})\,\tilde{\cW}_{g_2,|M_2|,|K_2|}(w_{M_2}\mid z_{K_2}).
\end{align}
Note that after one applies $\restr_{w = w_0} \Delta_{w_0} \Delta_w$ to \eqref{eq:gen2loop}, the first term in the RHS, as well as the first factors in the terms in the sum in the second line of the equation, are holomorphic in $w_0$, due to our assumption that quadratic loop equations are satisfied up to negative Euler characteristic $-\chi$. And the second factors in the terms in the sum in the second line are constant in $w_0$. Thus, the whole expression is holomorphic in $w_0$ near ramification points.
\end{proof}

Now we are ready to prove the following proposition, which is the main technical result of the present paper:
\begin{proposition}\label{AllDeltasHolomorphic}
	Suppose that a set of functions $ (\tilde W_{g,n} )_{g,n}$ satisfies the quadratic loop equations up to negative Euler characteristic $-\chi$. Then, we get for any $N, g, n \geq 0$ such that $ 2g-2+n \leq -\chi$
	, that
	\begin{equation}\label{ExtendedQLE}
		\sum_{k=0}^N \frac{1}{(2k)!} \sum_{\alpha_1 + \dotsb \alpha_{2k} + k = N} \prod_{j=1}^{2k} \bigg( \restr_{w_j = z} \frac{D_j^{2\alpha_j}}{(2\alpha_j+1)!} \bigg) \Delta_1 \dotsb \Delta_{2k} \tilde{\cW}_{g-\alpha_1 -\dotsb -\alpha_{2k},2k,n} (w_1, \dotsc , w_{2k} \mid z_{\llbracket n\rrbracket}) \,,
	\end{equation}
	where $D_j \coloneqq \frac{d}{d X(w_j)}$, is holomorphic in $z$ near branch points of the spectral curve.
\end{proposition}
\begin{proof}
	We use induction on $N$ and $g$.
	First note that $k=0$ can only occur if $ N=0$, and in this case, the statement is trivial, as the expression is constant in $z$.\par
	For $N=1$, the statement is just the quadratic loop equation, which holds by assumption, and furthermore, for $g=-1$ it is clearly zero.

	Let us define, for $s\geq 0$, 
	\begin{align}
	&\Hol_{g,N,n,s}(z,\tilde w_{\llbracket s \rrbracket}) \coloneqq \\ \nonumber
	&\sum_{k=0}^N \frac{1}{(2k)!} \sum_{\alpha_1 + \dotsb \alpha_{2k} + k = N} \prod_{j=1}^{2k} \bigg( \restr_{w_j = z} \frac{D_j^{2\alpha_j}}{(2\alpha_j+1)!} \bigg) \Delta_1 \dotsb \Delta_{2k} \tilde{\cW}_{g-\alpha_1 -\dotsb -\alpha_{2k},2k+s,n} (w_{\llbracket 2k \rrbracket},\tilde w_{\llbracket s \rrbracket} \mid z_{\llbracket n\rrbracket}).
	\end{align}
	(we omit the dependence on $z_{\llbracket n\rrbracket}$ in the LHS for brevity).
	
	Now let us fix some $N_0$ and $g_0$ and suppose that the statement of the proposition, which can now be rephrased as
	\begin{equation}	
	\Hol_{g,N,n,0}(z) \text{ is holomorphic in $z$ near ramification points,}
	\end{equation}	
	holds for all $(g,N,n)$ such that
	\begin{align}\label{eq:gNnconds}
	g &\leq g_0+1,\\ \nonumber
	N &\leq N_0,\\ \nonumber
	n &\leq -\chi +2-2g.
	\end{align}
	If we, under these assumptions, manage to prove the statement for $N=N_0+1,\,g=g_0+1$ (and for all $n \leq -\chi-2 g_0$), we will, by induction, achieve our goal (since, as explained above, the statement holds at the boundaries $N=1$ and $g=-1$).
	
	Note that under these assumptions we have the following statement:
	\begin{equation}\label{eq:GenExtendedQLE}
	\Hol_{g,N,n,s}(z,\tilde w_{\llbracket s \rrbracket}) \text{ is holomorphic in $z$ near ramification points,}
	\end{equation}
	for the same $(g,N,n)$ as in \eqref{eq:gNnconds} and all $s\geq 0$.
	The proof of this statement is completely analogous to the proof of proposition \ref{prop:gen2loop}.
	

	For brevity from now on we write $(g,N)$ in place of $(g_0,N_0)$. We will express $\Hol_{g+1,N+1,n,0}(z)$ in terms of previous cases.
	First of all, we take
	\begin{equation}\label{D2Hol}
		\begin{split}
			&D^2_{x(z)} \Hol_{g,N,n,0} (z)\\ 
			=& \sum_{k=1}^N \frac{1}{(2k)!} \sum_{\alpha_1 + \dotsb +\alpha_{2k} + k = N} \sum_{i=1}^{2k} \prod_{j=1}^{2k} \bigg(\restr_{w_j = z} \frac{ D_j^{2\alpha_j + 2 \delta_{ij}}}{(2\alpha_j+1)!} \bigg)\; \prod_{i=1}^{2k} \Delta_{w_i} \tilde{\cW}_{g-N+k,2k,n} (w_{\llbracket 2k \rrbracket})\\
			& + \sum_{k=0}^N \frac{1}{(2k)!} \binom{2k}{2} \sum_{\substack{\alpha_1 + \dotsb +\alpha_{2k-2} \\+ \beta_1 + \beta_2 + k = N}} \prod_{j=1}^{2k-2} \bigg(\restr_{w_j = z} \frac{ D_j^{2\alpha_j}}{(2\alpha_j+1)!} \bigg)\\
			& \phantom{++}\times \restr_{w_{2k-1} = z}\frac{ D_{2k-1}^{2\beta_1+1}}{(2\beta_1+1)!} \restr_{w_{2k} = z}\frac{ D_{2k}^{2\beta_2 +1}}{(2\beta_2+1)!}\prod_{i=1}^{2k} \Delta_{w_i} \tilde{\cW}_{g-N+k,2k,n} (w_{\llbracket 2k \rrbracket})\\
			=& \sum_{k=1}^N \frac{1}{(2k)!} \sum_{\alpha_1 + \dotsb +\alpha_{2k} + k = N} \sum_{i=1}^{2k} \prod_{j=1}^{2k} \bigg(\restr_{w_j = z} \frac{ D_j^{2\alpha_j + 2 \delta_{ij}}}{(2\alpha_j+1)!} \bigg)\;\prod_{i=1}^{2k} \Delta_{w_i} \tilde{\cW}_{g-N+k,2k,n} (w_{\llbracket 2k \rrbracket})\\
			& + \sum_{\substack{N_\alpha + N_\beta = N\\N_\alpha\geq 0,\; N_\beta\geq 1}} \sum_{\beta_1 + \beta_2 + 1 = N_\beta} \restr_{w_{1'} = z}\frac{D_{1'}^{2\beta_1+1}}{(2\beta_1+1)!} \restr_{w_{2'} = z} \frac{D_{2'}^{2\beta_2 +1}}{(2\beta_2+1)!} \sum_{k=0}^{N_\alpha} \frac{1}{(2k)!}
			\\
			& \phantom{++}\times \sum_{\alpha_1 + \dotsb +\alpha_{2k} + k = N_\alpha} \prod_{j=1}^{2k} \bigg(\restr_{w_j = z} \frac{D_j^{2\alpha_j}}{(2\alpha_j+1)!} \bigg)\; \Delta_{w_{1'}}\Delta_{w_{2'}}\prod_{i=1}^{2k} \Delta_{w_i} \tilde{\cW}_{g-N+k+1,2k+2,n} (w_{1'},w_{2'},w_{\llbracket 2k \rrbracket}),
		\end{split}
	\end{equation}
	where we have omitted the $z_{\llbracket n \rrbracket}$ arguments of $\tilde\cW$ for brevity (and we will keep omitting them for the rest of this proof). $D_j$ here and from now on stands for $D_{x(w_j)}$.
	
	This whole expression is holomorphic in $z$, being the result of the application of  $\frac{d^2}{d X^2(z)}$ to an expression holomorphic in $z$. We also see that the terms in the third-to-last line in this equation are already of the form which we see in $\Hol_{g+1,N+1,n,0}$ 
	However, the terms corresponding to the second-to-last and the last lines contain odd derivatives in the second term, which are certainly absent from $\Hol_{g+1,N+1,n,0}$.
	To counteract these odd-derivative terms, we would like to subtract
	\begin{equation}\label{eq:N-beta-N-alpha}
	\begin{split}
		&\sum_{\substack{N_\alpha + N_\beta = N\\N_\alpha,N_\beta\geq 0}}\restr_{\tilde z = z}\frac{D_{x(\tilde z)}^{2N_\beta}}{(2N_\beta)!} \; \restr_{ w_{1'},w_{2'} = \tilde z}\Delta_{w_{1'}}\Delta_{w_{2'}}\Hol_{g+1-N_\beta,N_\alpha,n,2} (z,w_{1'},w_{2'}) \\
		&=\sum_{\substack{N_\alpha + N_\beta = N\\N_\alpha,N_\beta\geq 0}}\restr_{\tilde z = z} \frac{D_{x(\tilde z)}^{2N_\beta}}{(2N_\beta)!} \restr_{ w_{1'},w_{2'} = \tilde z}\sum_{k=0}^{N_\alpha} \frac{1}{(2k)!} \\
		&\times\sum_{\alpha_1 + \dotsb \alpha_{2k} + k = N_\alpha} \prod_{j=1}^{2k} \bigg(\restr_{w_{j} = z}  \frac{D_j^{2\alpha_j}}{(2\alpha_j+1)!} \bigg)\Delta_{w_{1'}}\Delta_{w_{2'}}\prod_{i=1}^{2k} \Delta_{w_i}\tilde{\cW}_{g+1-N+k,2k+2,n}(w_{1'},w_{2'},w_{\llbracket 2k\rrbracket}).
	\end{split}
	\end{equation}
Note that we include the $N_\beta=0$ terms.  

Proposition \ref{prop:gen2loop} and statement \eqref{eq:GenExtendedQLE}, under our induction assumption, imply that each expression 
\begin{equation}\label{eq:holobyinduction}
\frac{D_{x(\tilde z)}^{2N_\beta}}{(2N_\beta)!} \; \restr_{w_{1'},w_{2'} = \tilde z}\Delta_{w_{1'}}\Delta_{w_{2'}}\Hol_{g+1-N_\beta,N_\alpha,n,2} (z,w_{1'},w_{2'})
\end{equation}
is holomorphic in both $z$ and $\tilde z$ separately (at the ramification points), \emph{once we do not apply the convention \eqref{DeltaDeltaBreg} to} $\Delta_{w_{i'}}\Delta_{w_{j}} \tilde W_{0,2}(w_{i'},w_j)$, $i'=1',2'$, $j=1,\dots,2k$. Note that in order to claim this, as per the conditions of proposition \ref{prop:gen2loop} and statement \eqref{eq:GenExtendedQLE}, we have to restrict $n$. Namely, for the holomorphicity in $z$ we need the condition $n\leq -\chi + 2-2(g+1-N_\beta)$ to hold for all $0\leq N_\beta\leq N$, and for the holomorphicity in $\tilde z$ we need the condition $n\leq -\chi + 2-2(g+1-N + k)$ to hold for all $0\leq k \leq N$. Both of these conditions are equivalent to $n\leq -\chi -2g$, which is precisely what want for our induction step.
	
Remarkably, after the application of $\restr_{\tilde z = z}$ expression~\eqref{eq:holobyinduction} remains holomorphic in $z$. In order to see this, let us prove that 
\begin{align} \label{eq:pf-holo}
	&
	\Res_{\tilde z \to z} \frac{dX(\tilde z)}{X(\tilde z)-X(z)}
	\frac{D_{x(\tilde z)}^{2N_\beta}}{(2N_\beta)!} \; \restr_{w_{1'},w_{2'} = \tilde z}\Delta_{w_{1'}}\Delta_{w_{2'}}\Hol_{g+1-N_\beta,N_\alpha,n,2} (z,w_{1'},w_{2'})
	\\
	&
	\notag
	= \frac 12	\int_{|\tilde z|=\epsilon} \frac{dX(\tilde z)}{X(\tilde z)-X(z)}
	\frac{D_{x(\tilde z)}^{2N_\beta}}{(2N_\beta)!} \; \restr_{w_{1'},w_{2'} = \tilde z}\Delta_{w_{1'}}\Delta_{w_{2'}}\Hol_{g+1-N_\beta,N_\alpha,n,2} (z,w_{1'},w_{2'}),
\end{align}
for $|z|<\epsilon$, where we assume that $\epsilon$ is a fixed number. Note two properties of the expression under the sign of the integral on the right hand side of equation~\eqref{eq:pf-holo}:
\begin{enumerate}
	\item its only poles in $\tilde z$ are at $\tilde z = z$ and $\tilde z = -z$, and the residues at these two poles are equal to each other by the symmetry of this expression under the sign change;
	\item it is holomorphic in $z$ for $|\tilde z| = \epsilon$ and $|z|<\epsilon$. 
\end{enumerate}
The first property implies that equation~\eqref{eq:pf-holo} holds, the second property implies that the whole expression is holomorphic in $z$.

However, we want to use expression~\eqref{eq:N-beta-N-alpha} assuming the convention \eqref{DeltaDeltaBreg} for the possible factors  $\Delta_{w_{i'}}\Delta_{w_{j}} \tilde W_{0,2}(w_{i'},w_j)$, $i'=1',2'$, $j=1,\dots,2k$, for each $k$. In this way, it is not holomorphic, but by the previous paragraph it becomes holomorphic if we add the following terms:
\begin{align}\label{eq:addedterms-hol}
 & \sum_{j=1}^{2k}\frac{1}{(2N_\beta )!} \restr_{w_{1'} = w_j}D_{1'}^{2N_\beta}  \frac{D_j^{2\alpha_j}}{(2\alpha_j+1)!} \frac{2 }{(X(w_{1'})-X(w_{j}))^2} \restr_{w_{2'}=w_{1'}}  \Delta_{w_{2'}}\prod_{i\not=j} \Delta_{w_i}\tilde{\cW}_{g-N+k,2k,n}(w_{2'} , w_{\llbracket 2k \rrbracket \setminus \{j\} } ) \\ \notag
  & +\sum_{j=1}^{2k}\frac{1}{(2N_\beta )!} \restr_{w_{2'} = w_j}D_{1'}^{2N_\beta}  \frac{D_j^{2\alpha_j}}{(2\alpha_j+1)!} \frac{2 }{(X(w_{2'})-X(w_{j}))^2} \restr_{w_{1'}=w_{2'}}  \Delta_{w_{2'}}\prod_{i\not=j} \Delta_{w_i}\tilde{\cW}_{g-N+k,2k,n}(w_{1'} , w_{\llbracket 2k \rrbracket \setminus \{j\} } ) \\ \notag
  & + \sum_{l\not=j} \frac{1}{(2N_\beta )!} \restr_{w_{1'} = w_j}D_{1'}^{2N_\beta} \restr_{w_{l} = w_j} \frac{D_l^{2\alpha_l}}{(2\alpha_l+1)!} \frac{2 }{(X(w_{1'})-X(w_{l}))^2}
  \times
  \\ \notag & \hspace{1cm}\frac{D_j^{2\alpha_j}}{(2\alpha_j+1)!} \frac{2 }{(X(w_{1'})-X(w_{j}))^2} 
  \prod_{i\not=l,j} \Delta_{w_i}\tilde{\cW}_{g-N+k-1,2k-2,n}( w_{\llbracket 2k \rrbracket \setminus \{l,j\} } )
\end{align}
(we write these terms omitting the $D$-operators acting on the $w$'s which didn't appear in the $\tilde W_{0,2}(w_{i'},w_j)$ factors and the sum over $k$).
The sum of the first two summands in this expression is equal to
	\begin{equation}\label{eq:nonholterms}
		\begin{split}
			& \frac{2}{(2N_\beta )!} \restr_{w_{1'} = w_j}D_{1'}^{2N_\beta} \frac{D_j^{2\alpha_j}}{(2\alpha_j+1)!} \frac{2 }{(X(w_{1'})-X(w_{j}))^2} \Delta_{w_{1'}}\prod_{i\not=j} \Delta_{w_i}\tilde{\cW}_{g-N+k,2k,n}(w_{1'} , w_{\llbracket 2k \rrbracket \setminus \{j\} } ))\\
			&= \frac{4}{(2N_\beta)!(2\alpha_j+1)!} \Res_{w_{1'} \to w_j} \frac{dX(w_{1'})}{X(w_{1'})-X(w_j)} 
			\times \\
			 & \hspace{1cm}
			  D_{1'}^{2N_\beta} D_j^{2\alpha_j} \frac{1 }{(X(w_{1'})-X(w_{j}))^2} 
			 \Delta_{w_{1'}}\prod_{i\not=j} \Delta_{w_i}\tilde{\cW}_{g-N+k,2k,n}(w_{1'} , w_{\llbracket 2k \rrbracket \setminus \{j\} } ))\\
			&= \frac{4}{(2N_\beta)!} \Res_{w_{1'} \to w_j} D_{1'}^{2N_\beta}  \bigg(\frac{1}{X(w_{1'})-X(w_j)}\bigg) \frac{dX(w_{1'})}{(X(w_{1'})-X(w_j))^{2+2\alpha_j}} 
			 \Delta_{w_{1'}}\prod_{i\not=j} \Delta_{w_i}\tilde{\cW}_{g-N+k,2k,n}(w_{1'} ,w_{\llbracket 2k \rrbracket \setminus \{j\} } ))\\
			&= 4 \Res_{w_{1'} \to w_j} \frac{dX(w_{1'})}{(X(w_{1'})-X(w_j))^{3+2N_\beta + 2\alpha_j}} \Delta_{w_{1'}}\prod_{i\not=j} \Delta_{w_i}\tilde{\cW}_{g-N+k,2k,n}(w_{1'} ,w_{\llbracket 2k \rrbracket \setminus \{j\} } ) )\\
			&= 4 \frac{D_j^{2\alpha_j+2N_\beta +2}}{(2\alpha_j + 2N_\beta +2)!} \prod_{i=1}^{2k} \Delta_{w_i}\tilde{\cW}_{g-N+k,2k,n}(w_{\llbracket 2k \rrbracket } ) \,.
		\end{split}
	\end{equation}
By the same computation, the last summand in~\eqref{eq:addedterms-hol} is equal to zero.
 
	Thus, if we add all the terms corresponding to \eqref{eq:nonholterms} to \eqref{eq:N-beta-N-alpha}, we get a holomorphic expression, which is then equal to 
	\begin{equation}
		\begin{split}
		&\sum_{\substack{N_\alpha + N_\beta = N\\N_\alpha,N_\beta\geq 0}}\bigg(\restr_{\tilde z = z} \frac{D_{x(\tilde z)}^{2N_\beta}}{(2N_\beta)!} \restr_{\tilde w_{1'},w_{2'} = \tilde z}\sum_{k=0}^{N_\alpha} \frac{1}{(2k)!} \\
		&\times\sum_{\alpha_1 + \dotsb \alpha_{2k} + k = N_\alpha} \prod_{j=1}^{2k} \bigg(\restr_{w_{j} = z}  \frac{D_j^{2\alpha_j}}{(2\alpha_j+1)!} \bigg)\Delta_{w_{1'}}\cdots \Delta_{w_{2k}}\tilde{\cW}_{g+1-N+k,2k+2,n}(w_{1'},w_{2'},w_{\llbracket 2k\rrbracket})\\
		&+ 4 \sum_{k=1}^{N_\alpha} \frac{1}{(2k)!} \sum_{\alpha_1 + \dotsb \alpha_{2k} + k = N_\alpha} \sum_{i=1}^{2k} \restr_{w_{i} = z}\frac{D_i^{2\alpha_i+2N_\beta +2}}{(2\alpha_i + 2N_\beta +2)!}\prod_{\substack{j=1\\ j \neq i}}^{2k} \bigg(  \restr_{w_{j} = z}\frac{D_j^{2\alpha_j}}{(2\alpha_j+1)!} \bigg)\prod_{i=1}^{2k} \Delta_{w_i}\tilde{\cW}_{g-N+k,2k,n}(w_{\llbracket 2k \rrbracket } )\bigg)\,.
		\end{split}
	\end{equation}
	Subtracting expression~\eqref{D2Hol} (which itself is holomorphic) from this, we get (note that the index $N_\beta$ has been shifted here)
	\begin{equation}\label{ClaimForHolN+1}
		\begin{split}
			&\sum_{\substack{N_\alpha + N_\beta = N+1\\N_\beta \geq 1}}\Bigg( \sum_{\beta_1 + \beta_2 + 1 = N_\beta}   \restr_{w_{1'}=z}\restr_{w_{2'} = z}\frac{D^{2\beta_1}_{1'}}{(2\beta_1)!} \frac{D^{2\beta_2}_{2'}}{(2\beta_2)!} \sum_{k=0}^{N_\alpha} \frac{1}{(2k)!} \sum_{\substack{\alpha_1 + \cdots+ \alpha_{2k}\\ + k = N_\alpha}} \prod_{j=1}^{2k} \bigg(\restr_{w_{j} = z} \frac{D_j^{2\alpha_j}}{(2\alpha_j+1)!} \bigg)
			\\
			& \hspace{30mm}
			\prod_{i=1}^{2k} \Delta_{w_i} \tilde{\cW}_{g+1-N+k,2k+2,n}(w_{1'},w_{2'},w_{\llbracket 2k \rrbracket } ) \Bigg)\\
			&
			+ 4 \sum_{\substack{N_\alpha + N_\beta = N+1\\N_\beta \geq 1}}\Bigg( \sum_{k=1}^{N_\alpha} \frac{1}{(2k)!} \sum_{\alpha_1 + \dotsb \alpha_{2k} + k = N_\alpha} \sum_{i=1}^{2k} \restr_{w_{i} = z}\frac{D_i^{2\alpha_i+2N_\beta}}{(2\alpha_i + 2N_\beta)!}\prod_{\substack{j=1\\ j \neq i}}^{2k} \bigg(  \restr_{w_{j} = z}\frac{D_j^{2\alpha_j}}{(2\alpha_j+1)!} \bigg)
			\\
			& \hspace{30mm}
			\prod_{i=1}^{2k} \Delta_{w_i}\tilde{\cW}_{g-N+k,2k,n}(w_{\llbracket 2k \rrbracket } )\Bigg)\\
			&-\sum_{k=1}^N \frac{1}{(2k)!} \sum_{\alpha_1 + \dotsb +\alpha_{2k} + k = N} \sum_{i=1}^{2k} \prod_{j=1}^{2k} \bigg( \restr_{w_{j} = z}\frac{D_j^{2\alpha_j + 2 \delta_{ij}}}{(2\alpha_j+1)!} \bigg)\prod_{i=1}^{2k} \Delta_{w_i}\tilde{\cW}_{g-N+k,2k,n}(w_{\llbracket 2k \rrbracket } )\,,
		\end{split}
	\end{equation}
	which is holomorphic. 

	We claim that, up to a factor, this equals $ \Hol_{g+1,N+1,n,0}(z)$. Indeed, let us extract the coefficient of a term
	\begin{equation}
	\frac{1}{(2k)!} \prod_{j=1}^{2k}\restr_{w_{j} = z} \frac{D_j^{2\alpha_j}}{(2\alpha_j+1)!}\prod_{i=1}^{2k} \Delta_{w_i}\tilde{\cW}_{g+1-\sum\alpha_i,2k,n}(w_{\llbracket 2k \rrbracket } ),
	\end{equation}
	where $\alpha_1+\cdots+\alpha_{2k}+k=N+1$ (note that all terms in ~\eqref{ClaimForHolN+1} are of this form and satisfy this condition).
	From the first, second, and third summands of expression~\eqref{ClaimForHolN+1} we get, respectively
	\begin{align}
		2 &\sum_{1 \leq i < j \leq 2k} (2\alpha_i +1)(2\alpha_j+1)\,;\\ \nonumber
		4 &\sum_{i = 1}^{2k} (2\alpha_i + 1)\cdot \alpha_i\,;\\ \nonumber
		- &\sum_{i=1}^{2k} (2\alpha_i ) (2\alpha_i+1)\,;
	\end{align}
	where the $\alpha_i$ on the second line comes from the number of different ways of choosing $ N_\beta$. Adding up these terms, we get
	\begin{equation}
		\begin{split}
			& 2 \sum_{1 \leq i < j \leq 2k} (2\alpha_i +1)(2\alpha_j+1)+ 4 \sum_{i = 1}^{2k} (2\alpha_i + 1)\cdot \alpha_i - \sum_{i=1}^{2k} (2\alpha_i ) (2\alpha_i+1)\\
			&=  \sum_{1 \leq i \neq  j \leq 2k} (2\alpha_i +1)(2\alpha_j+1)+ \sum_{i = 1}^{2k} 2\alpha_i (2\alpha_i+ 1)\\
			& = \bigg( \sum_{i=1}^{2k} 2\alpha_i +1 \bigg)^2 - \sum_{i=1}^{2k} (2\alpha_i +1)\\
			&= (2N+2)^2 - (2N+2) = (2N+2)(2N+1)\,.
		\end{split}
	\end{equation}
	As this factor is independent of $k$ and the $\alpha_j$, this shows that expression~\eqref{ClaimForHolN+1} is equal to this factor times $\Hol_{g+1,N+1,n,0}$.
	Since expression~\eqref{ClaimForHolN+1} is holomorphic, and this whole reasoning works for any $n\leq -\chi -2g$, this proves the induction step and thus the proposition.
\end{proof}

\begin{remark}\label{rem:inductionEulerChar} In the induction step in the proof of proposition~\ref{AllDeltasHolomorphic} for $\Hol_{g+1,N+1,n,0}$ we used $\Hol_{g+1,i,n,0}$, $i=1,\dots, N$ for the same $g+1$ case. It is easy to trace through the proof all instances where these terms occur: they always come from expression~\ref{eq:N-beta-N-alpha} for $N_\beta=0,\; N_\alpha=k$. Applying the same induction argument, we obtain the following refinement of the statement of proposition~\ref{AllDeltasHolomorphic}: \emph{if the quadratic loop equations are satisfied up to the negative Euler characterteristic strictly less than $2g-2+n$, then for any $N\geq 1$ the following expression}
\begin{align}
& \sum_{k=0}^N \frac{1}{(2k)!} \sum_{\alpha_1 + \dotsb \alpha_{2k} + k = N} \prod_{j=1}^{2k} \bigg( \restr_{w_j = z} \frac{D_j^{2\alpha_j}}{(2\alpha_j+1)!} \bigg) \Delta_1 \dotsb \Delta_{2k} \tilde{\cW}_{g-\alpha_1 -\dotsb -\alpha_{2k},2k,n} (w_1, \dotsc , w_{2k} \mid z_{\llbracket n\rrbracket}) \\ \nonumber
& - \frac{1}{(2N)!} \binom{N}{1} \restr_{w_1 = z} \restr_{w_2 = z} \Delta_1 \Delta_{2} \tilde{\cW}_{g,2,n} (w_1, w_{2} \mid z_{\llbracket n\rrbracket}) \left(\Delta_z W_{0,1}(z) \right)^{2N-2}
\end{align}
\emph{is holomorphic}.
\end{remark}

\subsection{Quadratic loop equations from the cut-and-join equation}

We prove the quadratic loop equations (in the form~\eqref{eq:QLEDeltaFormulation}) by induction from the cut-and-join equation~\eqref{CutJoinConcise}. We distribute $\mc{S}$'s and $\Delta$'s in cut-and-join equation according to equation~\eqref{Sondiagonal} and express the result in terms of the form~\eqref{ExtendedQLE} with added $\mc{S}$'s. Then, we use inductive arguments both on the negative Euler characteristic $ 2g-2+n$ and on the number of $\Delta$'s involved. In fact, we prove that any particular instance of the cut-and-join equation, so for any choice of $ r,g,n$, is a combination of derivatives of linear and quadratic loop equations (for the same $r$), whose negative Euler characteristic is bounded from above by $ 2g-2+n$, and where the $ 2g-2+n$ quadratic loop equation occurs without derivatives and with a non-trivial coefficient. As the symmetrized cut-and-join equation is holomorphic and all the previous quadratic loop equations hold by induction, just as all linear loop equations, this will then prove the $(g,n)$ quadratic loop equation holds.\par
By distributing the $\mc{S}$'s and $\Delta$'s, we will always get an even number of $\Delta$'s. Hence, up to diagonal poles, we can always write such a distribution as a product of linear and quadratic loop equations. By the discussion above, there are no possible diagonal poles between two $\mc{S}$'s or between an $\mc{S}$ and a $\Delta$, so we should focus our attention on the $ \Delta$ factors.\par
Recall, from \eqref{CutJoinConcise}, that the symmetrized cut-and-join equation implies that
\begin{equation}
\mc{S}_{z_0} \sum_{\substack{m \geq 1, d \geq 0 \\ m + 2d = r+1}} \!\! \frac{1}{m!} \, \bar{Q}_{d,m}(z_0) \tilde{\cW}_{g-d,m,n}(w_{\llbracket m\rrbracket} \mid z_{\llbracket n\rrbracket})
\end{equation}
is holomorphic. Here (we recall the definitions for the reader's convenience)
\begin{align}
\sum_{d \geq 0} \bar{Q}_{d,m}(z_0) t^{2d} &= \frac{t}{\zeta(t)} \frac{\zeta(tD_{x(z_0)})}{tD_{x(z_0)}} \circ \prod_{j = 1}^m \bigg( \restr_{w_j = z_0} \circ \frac{\zeta(tD_{x(w_j)})}{tD_{x(w_j)}}\bigg)\,;\label{TaylorZeta}\\ \nonumber
\frac{\zeta (t)}{t} &= \frac{e^{t/2} - e^{-t/2}}{t} = \sum_{k=0}^\infty \frac{1}{(2k+1)!2^{2k}}t^{2k}\,;\\ \nonumber
\restr_{w=z} F(w) &\coloneq \Res_{w=z} F(w)\frac{dx(w)}{x(w)-x(z)}\,.
\end{align}
By our induction argument, we can omit any non-trivial contribution from $\frac{t \, dx(z_0)}{\zeta(t)} \frac{\zeta(tD_{x(z_0)})}{tD_{x(z_0)}}$, as it gives only a number of derivatives acting on symmetric terms that have inductively already been proved to be holomorphic. 

Recall also proposition~\ref{AllDeltasHolomorphic}. 
	In that proposition, the $2k$ and $ 2N$ are reminiscent of, respectively, $m$ and $ r+1$ in the cut-and-join equation, and they are written this way as we always have an even number of $\Delta$'s ($2k$) and an even number of $D$'s ($2N-2k$), the genus defect also being $N-k = \sum \alpha_i$. However, this proposition is only about the $\Delta$ part of any term, and it should still be multiplied with an $\mc{S}$ part.\par
	Furthermore, note that in proposition~\ref{AllDeltasHolomorphic} we have omitted the factors $\frac{1}{2}$ coming from equations~\eqref{Sondiagonal} and~\eqref{TaylorZeta}. As these give one factor for each $\Delta $ and $D$, respectively, and the sum of their exponents is constantly equal to $N$ in equation~\eqref{ExtendedQLE}, we may as well omit them. 
	
Proposition~\ref{AllDeltasHolomorphic} implies the following corollary.
	
\begin{corollary}\label{AllHolomorphic}
	Suppose that a set of functions $ (W_{g,n} )_{g,n}$ satisfies the quadratic loop equations up to negative Euler characteristic $-\chi$. Then, we get for any $r > 0$ and any $g, l,n \geq 0$ such that $ r+1 - l$ is even and $ 2g-2+n \leq -\chi$, that 
	\begin{equation}\label{eq:DoubleExtendedQLE}
	\sum_{\substack{m=l\\m-l \textup{ even}}}^{r+1} \frac{1}{m!} \sum_{2\alpha_1 + \dotsb 2\alpha_m + m = r+1} \prod_{j=1}^m \bigg( \restr_{z_j = z} \frac{D_j^{2\alpha_j}}{(2\alpha_j+1)!} \bigg) \sum_{\substack{I \subset \llbracket m\rrbracket\\ |I| =l}}\prod_{i\in \llbracket m\rrbracket  \setminus I}\!\!\! \Delta_{j}\prod_{i \in I} \mc{S}_i \, \tilde{\cW}_{g- \alpha_1 - \dotsc - \alpha_m,m,n} (w_{\llbracket m\rrbracket} \mid z_{\llbracket n\rrbracket}) 
	\end{equation}
	is holomorphic in $z$ near branch points of the spectral curve.
\end{corollary}
\begin{proof}
	For $ l =0$, this is a reformulation of proposition~\ref{AllDeltasHolomorphic}, with $ 2k = m$ and $ 2N = r+1$.\par
	In general we can rewrite it, by reshuffling, as
	\begin{equation}
	\begin{split}
	& \sum_{\substack{N = 0\\N \textup{ even}}}^{r+1} \sum_{k=0}^{N} \, \bigg[
	\frac{1}{l!} \sum_{\substack{2\beta_1 + \dotsb + 2\beta_l + l \\ = r+1-N}}\prod_{j=1}^l \bigg( \restr_{w'_j = z} \frac{D_{j'}^{2\beta_j}}{(2\beta_j+1)!} \bigg) \prod_{i'=1}^l \mc{S}_{i'} \bigg] 
	\\
	&
	\bigg[\frac{1}{(2k)!}\sum_{\alpha_1 + \dotsb \alpha_{2k} + k = N} \prod_{j=1}^{2k} \bigg( \restr_{w_j = z} \frac{D_j^{2\alpha_j}}{(2\alpha_j+1)!} \bigg) 
	\prod_{i =1}^{2k} \Delta_i \bigg] \ \tilde{\cW}_{g-  \beta_1 - \dotsc \beta_l -\alpha_1 - \dotsc - \alpha_{2k},2k+l,n} ( w'_{\llbracket l\rrbracket }, w_{\llbracket 2k\rrbracket} \mid z_{\llbracket n\rrbracket }  ) \,
	\end{split}
	\end{equation}
	(in order to shorten the notation we use $D_{j'}\coloneqq D_{x(w'_j)}$ and $D_j \coloneqq D_{x(w_j)})$).
	In this formula, for a fixed choice of $N$, the $k$- and $\alpha$-sums give something holomorphic by proposition~\ref{AllDeltasHolomorphic}, the extra $\mc{S}$'s do not change holomorphicity by the linear loop equations and the fact that $\mc{S}_{i'}\mc{S}_{j'} \tilde{W}_{0,2}(w'_{i},w'_j)$ respectively $ \mc{S}_{i'}\Delta_j \tilde{W}_{0,2}(w'_i,w_j)$ are holomorphic at the diagonal, and the operator $D_{i'}^{2\beta_i}D_{j'}^{2\beta_j}$ respectively $D_{i'}^{2\beta_i}$ do not change that.
\end{proof}

\begin{theorem} \label{thm:QuadraticLoopForSpinHurwitz}
	The quadratic loop equations~\eqref{eq:QLEDeltaFormulation} hold for $ ({W}_{g,n})_{g,n}$ in the case of $r$-spin Hurwitz numbers, i.e. for
	\begin{equation}
	W_{g,n} - \delta_{g,0}\delta_{n,2} \frac{1}{(X_1-X_2)^2} \sim\  \prod_{i=1}^n \Big(\frac{d}{dX_i} \Big) \sum_{\mu_1,\dots,\mu_n=1}^\infty h_{g;\mu}^{\circ, q,r} \prod_{i=1}^n e^{X_i\mu_i}\,.
	\end{equation}
\end{theorem}
\begin{proof}
	As stated before, we use induction on the negative Euler characteristic.\par
	So assume the quadratic loop equation has been proved up to $ -\chi$, and consider the symmetrized cut-and-join equation for $ 2g-2 +n = -\chi +1$. All the sub-leading terms in the cut-and-join equation, i.e., those where $ \bar{Q}_{d,m}(z_0) $ gives a non-trivial contribution from $\frac{t}{\zeta(t)} \frac{\zeta(tD_{x(z_0)})}{tD_{x(z_0)}}$, are already holomorphic by the induction hypothesis, equation~\eqref{Sondiagonal}, and corollary~\ref{AllHolomorphic}. In the leading term, by the same corollary (cf. also~remark~\ref{rem:inductionEulerChar}), everything is holomorphic, except possibly for the terms involving
	\begin{equation}
	\restr_{w_1=z_0} 	\restr_{w_2=z_0} \Delta_{w_1} \Delta_{w_2} \tilde{\cW}(w_1,w_2 \mid z_{\llbracket n\rrbracket}) \cdot \mc{S}_{z_0} \left(y(z_0)^{r-1}\right)
	\end{equation}
	(as $ W_{0,1}(z_0) = y(z_0)$).\par
	Hence, this term must be holomorphic as well, and because $y(z)$ (and hence $\mc{S}_z y(z)$) is non-zero at branchpoints of $x$, this shows
	\begin{equation}
	\restr_{w_1=z_0} 	\restr_{w_2=z_0} \Delta_{w_1} \Delta_{w_2} \tilde{\cW}(w_1,w_2 \mid z_{\llbracket n\rrbracket}) \textup{ is holomorphic,}
	\end{equation}
	which is exactly the quadratic loop equation.
\end{proof}

\begin{remark}
Note that this proof generalizes the proofs of~\cite[theorems 14 \& 15]{SpecialCases}. In particular, proposition~\ref{AllDeltasHolomorphic} subsumes~\cite[lemma 16]{SpecialCases}, although the proof is different.
\end{remark}

\bibliographystyle{alpha}

\bibliography{rspinlib}

\newcommand{\etalchar}[1]{$^{#1}$}
\begin{thebibliography}{DBKO{\etalchar{+}}15}

\bibitem[ACEH18a]{AlChapuyProof}
A.~{Alexandrov}, G.~{Chapuy}, B.~{Eynard}, and J.~{Harnad}.
\newblock {Weighted Hurwitz numbers and topological recursion}.
\newblock {\em arXiv e-prints}, page arXiv:1806.09738, Jun 2018.

\bibitem[ACEH18b]{AlChapuy}
A.~Alexandrov, G.~Chapuy, B.~Eynard, and J.~Harnad.
\newblock Weighted {H}urwitz numbers and topological recursion: an overview.
\newblock {\em J. Math. Phys.}, 59(8):081102, 21, 2018.

\bibitem[AJ03]{AbramovichJarvis}
Dan Abramovich and Tyler~J. Jarvis.
\newblock Moduli of twisted spin curves.
\newblock {\em Proc. Amer. Math. Soc.}, 131(3):685--699, 2003.

\bibitem[Ale11]{Alexandrov}
A.~Alexandrov.
\newblock Matrix models for random partitions.
\newblock {\em Nucl. Phys. B}, 851(3):620--650, 2011.
\newblock hep-th/1005.5715.

\bibitem[ALS16]{ALS}
A.~Alexandrov, D.~Lewanski, and S.~Shadrin.
\newblock Ramifications of {H}urwitz theory, {KP} integrability and quantum
  curves.
\newblock {\em J. High Energy Phys.}, 2016(5):124, front matter+30, 2016.

\bibitem[BE13]{BouchardEynard}
Vincent Bouchard and Bertrand Eynard.
\newblock Think globally, compute locally.
\newblock {\em J. High Energy Phys.}, 2013(2):143, front matter + 34, 2013.

\bibitem[BEMS11]{BorotMulase}
Ga\"{e}tan Borot, Bertrand Eynard, Motohico Mulase, and Brad Safnuk.
\newblock A matrix model for simple {H}urwitz numbers, and topological
  recursion.
\newblock {\em J. Geom. Phys.}, 61(2):522--540, 2011.

\bibitem[BEO15]{BEO13}
Ga\"{e}tan Borot, Bertrand Eynard, and Nicolas Orantin.
\newblock Abstract loop equations, topological recursion and new applications.
\newblock {\em Commun. Number Theory Phys.}, 9(1):51--187, 2015.

\bibitem[BHSLM14]{BouchardMulase}
Vincent Bouchard, Daniel Hern\'{a}ndez~Serrano, Xiaojun Liu, and Motohico
  Mulase.
\newblock Mirror symmetry for orbifold {H}urwitz numbers.
\newblock {\em J. Differential Geom.}, 98(3):375--423, 2014.

\bibitem[BKL{\etalchar{+}}17]{SpecialCases}
Ga{\"e}tan {Borot}, Reinier {Kramer}, Danilo {Lewanski}, Alexandr {Popolitov},
  and Sergey {Shadrin}.
\newblock {Special cases of the orbifold version of Zvonkine's \$r\$-ELSV
  formula}.
\newblock {\em arXiv e-prints}, page arXiv:1705.10811, May 2017.

\bibitem[BMn08]{BouchardMarino}
Vincent Bouchard and Marcos Mari\~{n}o.
\newblock Hurwitz numbers, matrix models and enumerative geometry.
\newblock In {\em From {H}odge theory to integrability and {TQFT}
  tt*-geometry}, volume~78 of {\em Proc. Sympos. Pure Math.}, pages 263--283.
  Amer. Math. Soc., Providence, RI, 2008.

\bibitem[BS17]{BoSh15}
G.~Borot and S.~Shadrin.
\newblock Blobbed topological recursion: properties and applications.
\newblock {\em Math. Proc. Camb. Phil. Soc.}, 162(1):39--87, 2017.

\bibitem[CCC07]{CaporaseEtAl}
Lucia Caporaso, Cinzia Casagrande, and Maurizio Cornalba.
\newblock Moduli of roots of line bundles on curves.
\newblock {\em Trans. Amer. Math. Soc.}, 359(8):3733--3768, 2007.

\bibitem[CE06]{ChEy06}
Leonid Chekhov and Bertrand Eynard.
\newblock Matrix eigenvalue model: {F}eynman graph technique for all genera.
\newblock {\em J. High Energy Phys.}, (12):026, 29, 2006.

\bibitem[Chi08a]{Chiodo-stable-twisted}
Alessandro Chiodo.
\newblock Stable twisted curves and their {$r$}-spin structures.
\newblock {\em Ann. Inst. Fourier (Grenoble)}, 58(5):1635--1689, 2008.

\bibitem[Chi08b]{Chio08}
Alessandro Chiodo.
\newblock Towards an enumerative geometry of the moduli space of twisted curves
  and {$r$}th roots.
\newblock {\em Compos. Math.}, 144(6):1461--1496, 2008.

\bibitem[CJ18]{CladerJanda}
Emily Clader and Felix Janda.
\newblock Pixton's double ramification cycle relations.
\newblock {\em Geom. Topol.}, 22(2):1069--1108, 2018.

\bibitem[CR10]{ChiodoRuan}
Alessandro Chiodo and Yongbin Ruan.
\newblock Landau-{G}inzburg/{C}alabi-{Y}au correspondence for quintic
  three-folds via symplectic transformations.
\newblock {\em Invent. Math.}, 182(1):117--165, 2010.

\bibitem[DBKO{\etalchar{+}}15]{DKOSS13}
P.~Dunin-Barkowski, M.~Kazarian, N.~Orantin, S.~Shadrin, and L.~Spitz.
\newblock Polynomiality of {H}urwitz numbers, {B}ouchard-{M}ari\~{n}o
  conjecture, and a new proof of the {ELSV} formula.
\newblock {\em Adv. Math.}, 279:67--103, 2015.

\bibitem[DBLPS15]{DLPS}
P.~Dunin-Barkowski, D.~Lewanski, A.~Popolitov, and S.~Shadrin.
\newblock Polynomiality of orbifold {H}urwitz numbers, spectral curve, and a
  new proof of the {J}ohnson-{P}andharipande-{T}seng formula.
\newblock {\em J. Lond. Math. Soc. (2)}, 92(3):547--565, 2015.

\bibitem[DBOSS14]{DOSS14}
P.~Dunin-Barkowski, N.~Orantin, S.~Shadrin, and L.~Spitz.
\newblock Identification of the {G}ivental formula with the spectral curve
  topological recursion procedure.
\newblock {\em Comm. Math. Phys.}, 328(2):669--700, 2014.

\bibitem[DLN16]{Do-Orbifold}
Norman Do, Oliver Leigh, and Paul Norbury.
\newblock Orbifold {H}urwitz numbers and {E}ynard-{O}rantin invariants.
\newblock {\em Math. Res. Lett.}, 23(5):1281--1327, 2016.

\bibitem[DNO{\etalchar{+}}18]{DBNOPS-Primary}
P.~{Dunin-Barkowski}, P.~Norbury, N.~Orantin, A.~Popolitov, and S.~Shadrin.
\newblock Primary invariants of {H}urwitz {F}robenius manifolds.
\newblock In {\em Topological recursion and its influence in analysis,
  geometry, and topology}, volume 100 of {\em Proc. Sympos. Pure Math.}, pages
  297--331. Amer. Math. Soc., Providence, RI, 2018.

\bibitem[DNO{\etalchar{+}}19]{DBNOPS-Dubrovin}
P.~{Dunin-Barkowski}, P.~Norbury, N.~Orantin, A.~Popolitov, and S.~Shadrin.
\newblock Dubrovin's superpotential as a global spectral curve.
\newblock {\em J. Inst. Math. Jussieu}, 18(3):449--497, 2019.

\bibitem[DSS13]{DBSS-GivGraphs}
Petr {Dunin-Barkowski}, Sergey Shadrin, and Loek Spitz.
\newblock Givental graphs and inversion symmetry.
\newblock {\em Lett. Math. Phys.}, 103(5):533--557, 2013.

\bibitem[ELSV01]{ELSV01}
Torsten Ekedahl, Sergei Lando, Michael Shapiro, and Alek Vainshtein.
\newblock Hurwitz numbers and intersections on moduli spaces of curves.
\newblock {\em Invent. Math.}, 146(2):297--327, 2001.

\bibitem[EMS11]{EMS-Laplace}
Bertrand Eynard, Motohico Mulase, and Bradley Safnuk.
\newblock The {L}aplace transform of the cut-and-join equation and the
  {B}ouchard-{M}ari\~{n}o conjecture on {H}urwitz numbers.
\newblock {\em Publ. Res. Inst. Math. Sci.}, 47(2):629--670, 2011.

\bibitem[EO07]{EyOr07}
B.~Eynard and N.~Orantin.
\newblock Invariants of algebraic curves and topological expansion.
\newblock {\em Commun. Number Theory Phys.}, 1(2):347--452, 2007.

\bibitem[{Eyn}11]{Eyna11}
B.~{Eynard}.
\newblock {Intersection numbers of spectral curves}.
\newblock {\em arXiv e-prints}, page arXiv:1104.0176, Apr 2011.

\bibitem[Eyn14a]{Eynard2014}
B.~Eynard.
\newblock Invariants of spectral curves and intersection theory of moduli
  spaces of complex curves.
\newblock {\em Commun. Number Theory Phys.}, 8(3):541--588, 2014.

\bibitem[Eyn14b]{EynardSurvey}
Bertrand Eynard.
\newblock An overview of the topological recursion.
\newblock In {\em Proceedings of the {I}nternational {C}ongress of
  {M}athematicians---{S}eoul 2014. {V}ol. {III}}, pages 1063--1085. Kyung Moon
  Sa, Seoul, 2014.

\bibitem[GV03]{GraberVakil}
Tom Graber and Ravi Vakil.
\newblock Hodge integrals and {H}urwitz numbers via virtual localization.
\newblock {\em Compositio Math.}, 135(1):25--36, 2003.

\bibitem[Jar00]{Jarvis}
Tyler~J. Jarvis.
\newblock Geometry of the moduli of higher spin curves.
\newblock {\em Internat. J. Math.}, 11(5):637--663, 2000.

\bibitem[Joh15]{John15}
Paul Johnson.
\newblock Double {H}urwitz numbers via the infinite wedge.
\newblock {\em Trans. Amer. Math. Soc.}, 367(9):6415--6440, 2015.

\bibitem[JPPZ17]{JPPZ}
F.~Janda, R.~Pandharipande, A.~Pixton, and D.~Zvonkine.
\newblock Double ramification cycles on the moduli spaces of curves.
\newblock {\em Publ. Math. Inst. Hautes \'{E}tudes Sci.}, 125:221--266, 2017.

\bibitem[JPT11]{JohnsonPandharipandeTseng}
P.~Johnson, R.~Pandharipande, and H.-H. Tseng.
\newblock Abelian {H}urwitz-{H}odge integrals.
\newblock {\em Michigan Math. J.}, 60(1):171--198, 2011.

\bibitem[KLPS19]{KLPS17}
R.~Kramer, D.~Lewanski, A.~Popolitov, and S.~Shadrin.
\newblock Towards an orbifold generalization of {Z}vonkine's {$r$}-{ELSV}
  formula.
\newblock {\em Trans. Amer. Math. Soc.}, 372(6):4447--4469, 2019.

\bibitem[KLS19]{KLS}
Reinier Kramer, Danilo Lewanski, and Sergey Shadrin.
\newblock Quasi-polynomiality of monotone orbifold {H}urwitz numbers and
  {G}rothendieck's dessins d'enfants.
\newblock {\em Doc. Math.}, 24:857--898, 2019.

\bibitem[KM94]{KontsevichManin}
M.~Kontsevich and Yu. Manin.
\newblock Gromov-{W}itten classes, quantum cohomology, and enumerative
  geometry.
\newblock {\em Comm. Math. Phys.}, 164(3):525--562, 1994.

\bibitem[KO94]{KerovOlshanski1994}
S.~Kerov and G.~Olshanski.
\newblock Polynomial functions on the set of {Y}oung diagrams.
\newblock {\em C. R. Acad. Sci. Paris S\'er. I Math.}, 319(2):121--126, 1994.

\bibitem[{Lei}18]{Leigh}
Oliver {Leigh}.
\newblock {The Moduli Space of Stables Maps with Divisible Ramification}.
\newblock {\em arXiv e-prints}, page arXiv:1812.06933, Dec 2018.

\bibitem[Lew18]{LewanskiSurvey}
D.~Lewanski.
\newblock On {ELSV}-type formulae, {H}urwitz numbers and topological recursion.
\newblock In {\em Topological recursion and its influence in analysis,
  geometry, and topology}, volume 100 of {\em Proc. Sympos. Pure Math.}, pages
  517--532. Amer. Math. Soc., Providence, RI, 2018.

\bibitem[LPSZ17]{LPSZ-Chiodo}
Danilo Lewanski, Alexandr Popolitov, Sergey Shadrin, and Dimitri Zvonkine.
\newblock Chiodo formulas for the {$r$}-th roots and topological recursion.
\newblock {\em Lett. Math. Phys.}, 107(5):901--919, 2017.

\bibitem[MSS13]{MSS}
M.~Mulase, S.~Shadrin, and L.~Spitz.
\newblock The spectral curve and the {S}chr\"odinger equation of double
  {H}urwitz numbers and higher spin structures.
\newblock {\em Commun. Number Theory Phys.}, 7(1):125--143, 2013.

\bibitem[OP06]{OkPa06a}
A.~Okounkov and R.~Pandharipande.
\newblock Gromov-{W}itten theory, {H}urwitz theory, and completed cycles.
\newblock {\em Ann. of Math. (2)}, 163(2):517--560, 2006.

\bibitem[Ros08]{Rossi}
P.~Rossi.
\newblock Gromov-{W}itten invariants of target curves via symplectic field
  theory.
\newblock {\em J. Geom. Phys.}, 58(8):931--941, 2008.

\bibitem[SSZ12]{SSZ12}
S.~Shadrin, L.~Spitz, and D.~Zvonkine.
\newblock On double {H}urwitz numbers with completed cycles.
\newblock {\em J. Lond. Math. Soc. (2)}, 86(2):407--432, 2012.

\bibitem[SSZ15]{SSZ}
S.~Shadrin, L.~Spitz, and D.~Zvonkine.
\newblock Equivalence of {ELSV} and {B}ouchard-{M}ari\~{n}o conjectures for
  {$r$}-spin {H}urwitz numbers.
\newblock {\em Math. Ann.}, 361(3-4):611--645, 2015.

\bibitem[Vak03]{Vakil2003}
Ravi Vakil.
\newblock The moduli space of curves and its tautological ring.
\newblock {\em Notices Amer. Math. Soc.}, 50(6):647--658, 2003.

\bibitem[Vak08]{Vakil2008}
R.~Vakil.
\newblock The moduli space of curves and {G}romov-{W}itten theory.
\newblock In {\em Enumerative invariants in algebraic geometry and string
  theory}, volume 1947 of {\em Lecture Notes in Math.}, pages 143--198.
  Springer, Berlin, 2008.

\bibitem[Zvo06]{Zvonkine2006}
D.~Zvonkine.
\newblock A preliminary text on the \texorpdfstring{\( r\)}{r}-{ELSV} formula.
\newblock {\em Preprint}, 2006.

\bibitem[Zvo12]{Zvonkine2012}
Dimitri Zvonkine.
\newblock An introduction to moduli spaces of curves and their intersection
  theory.
\newblock In {\em Handbook of {T}eichm{\"u}ller theory. {V}olume {III}},
  volume~17 of {\em IRMA Lect. Math. Theor. Phys.}, pages 667--716. Eur. Math.
  Soc., Z{\"u}rich, 2012.

\end{thebibliography}

\end{document}